\documentclass[12pt]{amsart}

\title{The distribution of $a$-numbers of hyperelliptic curves in characteristic three}
\author{Derek Garton}
\address{Fariborz Maseeh Department of Mathematics and Statistics, Portland State University}
\email{\href{mailto:gartondw@pdx.edu}{gartondw@pdx.edu}}
\author{Jeffrey Lin Thunder}
\address{Department of Mathematical Sciences, Northern Illinois University}
\email{\href{mailto:jthunder@niu.edu}{jthunder@niu.edu}}
\author{Colin Weir}
\address{The Tutte Institute for Mathematics and Computing}
\email{\href{colinoftheweirs@gmail.com}{colinoftheweirs@gmail.com}}
\date{\today}

\usepackage{amssymb,
url,
MnSymbol,
enumitem,
color,
mathrsfs}
\usepackage{amsmath}
\usepackage{amsfonts}
\usepackage{amssymb}
\usepackage{colonequals}

\usepackage{hyperref}
\hypersetup{breaklinks=true,
colorlinks=true,
urlcolor=blue,
linkcolor=cyan
}
\usepackage[top=1in,bottom=1in,left=1in,right=1in]{geometry}
\usepackage[capitalise,nameinlink]{cleveref}

\makeatletter
\def\imod#1{\allowbreak\mkern8mu{(\operator@font mod}\,\,#1)}
\makeatother

\theoremstyle{plain}
\newtheorem{thm}{Theorem}[subsection]
\newtheorem{prop}[thm]{Proposition}
\newtheorem{lemma}[thm]{Lemma}
\newtheorem{conj}[thm]{Conjecture}
\newtheorem{cor}[thm]{Corollary}

\theoremstyle{remark}

\newtheorem{example}[thm]{Example}

\theoremstyle{definition}
\newtheorem{defn}[thm]{Definition}

\newtheorem*{namedthm}{\namedthmname}
\newcounter{namedthm}

\newcommand\numberthis{\refstepcounter{equation}\tag{\theequation}}

\newcounter{primes}



\DeclareMathOperator{\IC}{IC}
\DeclareMathOperator{\Aut}{Aut}
\DeclareMathOperator{\PGL}{PGL}

\DeclareMathOperator{\codim}{codim}

\DeclareMathOperator{\Span}{span}

\DeclareMathOperator{\Spec}{Spec}
\DeclareMathOperator{\Hom}{Hom}
\DeclareMathOperator{\Jac}{Jac}
\DeclareMathOperator{\gl}{GL}
\DeclareMathOperator{\der}{d\!}
\DeclareMathOperator{\ord}{ord}
\DeclareMathOperator{\Div}{Div}
\DeclareMathOperator{\divi}{div}

\newcommand{\lv}{\ensuremath{\left\vert}}
\newcommand{\rv}{\ensuremath{\right\vert}}
\newcommand{\lp}{\ensuremath{\left(}}
\newcommand{\rp}{\ensuremath{\right)}}
\newcommand{\lb}{\ensuremath{\left\{}}
\newcommand{\rb}{\ensuremath{\right\}}}

\newcommand{\A}{\ensuremath{\mathbb{A}}}
\newcommand{\FF}{\ensuremath{\mathbb{F}}}
\newcommand{\Z}{\ensuremath{\mathbb{Z}}}
\newcommand{\Fq}{\ensuremath{{\mathbb{F}_q}}}

\newcommand{\cA}{\ensuremath{\mathcal{A}}}
\newcommand{\cH}{\ensuremath{\mathcal{H}}}

\newcommand{\cP}{{\ensuremath{\mathcal{P}}}}
\newcommand{\cPe}{\cP _{\epsilon}}

\newcommand{\ba}{\mathbb{A}}

\newcommand{\fa}{\mathfrak{A}}
\newcommand{\fc}{\mathfrak{C}}

\newcommand{\bfa}{F_{\ba}}

\newcommand{\uo}{\mathbf{0}}
\newcommand{\ux}{\mathbf{x}}
\newcommand{\ua}{\mathbf{\alpha }}
\newcommand{\ub}{\mathbf{\beta }}

\newcommand{\mapf}{\phi}
\newcommand{\mapg}{\gamma}

\begin{document}

\begin{abstract}
In this paper we present a new approach to counting the proportion of hyperelliptic curves of genus $g$ defined 
over a finite field $\Fq$ with a given $a$-number. 
In characteristic three this method gives exact probabilities for curves of the form $Y^2=f(X)$ with 
$f(X)\in\Fq[X]$ monic and cubefree, probabilities that match the data presented by Cais et al.\ in previous work.
These results are sufficient to derive precise estimates (in terms of $q$) for these probabilities when 
restricting to squarefree $f$.
As a consequence, for positive integers $a$ and $g$ we show that the nonempty strata of the moduli space 
of hyperelliptic curves of genus $g$ consisting of those curves with $a$-number $a$ are of codimension $2a-1$.
This contrasts with the analogous result for the moduli space of abelian varieties in which the codimensions 
of the strata are $a(a+1)/2$.
Finally, our results allow for an alternative heuristic conjecture to that of Cais et al.;
one that matches the available data.
\end{abstract}

\maketitle
\tableofcontents

\section{Introduction}\label{intro}

Let $p$ be a prime, let $q$ be a power of $p$, and let $\Fq$ the finite field with $q$ elements. 
Let $X$ be transcendental over $\Fq$ so that $\Fq (X)$ is a 
field of rational functions.
If $C$ is a smooth algebraic curve over $\Fq$, the \emph{$a$-number} of $C$ is an invariant of the 
$p$-torsion group scheme $\Jac{\lp C\rp}[p]$ which we will denote by $a\lp C\rp$; it is defined as 
\[
a\lp C\rp
=\dim_{\Fq}{\lp\Hom{\lp\Spec_{\textsf{Grp}/\Fq}(\Fq[X]/X^p)
,\Jac{(C)}[p]\rp}\rp}.
\]
If $C$ has genus $g$, then $0\leq a(C)\leq g$ (see, e.g.,~\cite{O}).  When $g=1$, an elliptic curve is called supersingular if and only if $a(C)=1$. An elliptic curve is called ordinary otherwise.  Analogously, we say a curve $C$ is \emph{ordinary} if $a(C)=0$, as this is generic.  At the other extreme, in~\cite{ooo} Oort shows that a curve $C$ is isomorphic to a product of $g$ supersingular elliptic curves if and only if $a(C)=g$.  The purpose of this paper is to study the distribution of $a$-numbers of hyperelliptic curves to 
address such questions as: what is the probability a hyperelliptic curve is ordinary?

For any monic cubefree $f\in\Fq [X]$ let $H_f$ denote the (possibly singular) variety over $\Fq$ with affine equation
\[
Y^2=f(X).
\]
If $H_f$ has smooth model $H$ we will write $a\lp H_f\rp$ for $a(H)$.
In order to study the distribution of $a\lp H_f\rp$ as $f$ varies we adopt the following 
parameterizing sets: for $\epsilon\in\lb1,2\rb$ and $g\in\Z_{\geq0}$ let
\[
\cPe (g) = \lb f\in\FF_q[X]\mid f\text{ monic, cubefree,}\ \deg(f)=2g+\epsilon\rb.
\]
We often restrict ourselves to squarefree polynomials, so let
\[
\cPe '(g)
=\lb f\in\cPe (g)\mid f\text{ squarefree}\rb.
\]
Furthermore, since we are concerned with the proportions of cubefree and squarefree polynomials which yield varieties of given $a$-number, we set
\[
\mu_{\epsilon,g}(a)
=\frac{\lv\lb f\in\cPe (g)\mid a\lp H_f\rp=a\rb\rv}
{\lv\cPe (g)\rv}
\hspace{20px}\text{and}\hspace{20px}
\mu '_{\epsilon,g}(a)
=\frac{\lv\lb f\in\cPe '(g)\mid a\lp H_f\rp=a\rb\rv}
{\lv\cPe '(g)\rv}
\]
for non-negative integers $a$. We will explicitly compute $\mu _{\epsilon ,g}(a)$ in the case $p=3$. 
Specifically, we prove the
following. (Here and throughout we use the bracket notation $[\cdot ]$ for the greatest integer function.)

\subsection{Main results}\label{resultz}

In \cref{consequences}, we prove the following theorems.

%

\begin{thm}\label{bigone}
Suppose $p=3$ and $g,a$ are nonnegative integers.
If $g=0$, then
\[
\mu_{2,g}(a)=\mu_{1,g}(a)=
\begin{cases}
0
&\text{if }a>0\\
1
&\text{if }a=0,
\end{cases}
\]
so suppose $g\geq1$.

If $g\equiv 0\pmod 3$ and $\epsilon = 1,2$ or if $g\equiv 2\pmod 3$ and $\epsilon = 2$, then
\[
\mu _{\epsilon ,g}(a)=
\begin{cases}
1-q^{-1}&\text{if $a=0$,}\\
q^{-2a+1}\lp1-q^{-2}\rp&\text{if $0 < a < [(g+1)/3]$,}\\
q^{-2a+1}&\text{if $a=[(g+1)/3]$,}\\
0&\text{if $a>[(g+1)/3]$.}
\end{cases}
\]

On the other hand, if $g\equiv 1\pmod 3$ and $\epsilon = 1,2$ or if $g\equiv 2\pmod 3$ and $\epsilon =1$, then
\begin{itemize}
\item
if $g\leq 2$, then
\[
\mu_{\epsilon,g}(a)=
\begin{cases}
0
&\text{if }a>1\\
q^{-1}\lp\frac{1-q^{-1}}{1-q^{-2}}\rp
&\text{if }a=1,\\
\frac{1-q^{-1}}{1-q^{-2}}
&\text{if }a=0
\end{cases}
\]
\item
and if $g>3$, then
\[
\mu _{\epsilon ,g}(a)=
\begin{cases}
1-q^{-1}&\text{if $a=0$,}\\
q^{-2a+1}\lp1-q^{-2}\rp&\text{if $0<a< [(g-1)/3]$,}\\
q^{-2a+1}\lp\frac{1+q^{-1}-q^{-2}}{1+q^{-1}}\rp&\text{if $a=[(g-1)/3]$,}\\
q^{-2a+1}\lp\frac{1}{1+q^{-1}}\rp&\text{if $a=[(g-1)/3]+1$,}\\
0&\text{if $a>[(g-1)/3]+1$.}
\end{cases}
\]
\end{itemize}
\end{thm}

In some cases \cref{bigone} immediately implies exact probabilities for the $a$-numbers of squarefree polynomials.

\begin{cor}\label{actualvalues}
Suppose $p=3$ and $g,a\in\Z_{\geq0}$.
Write $\lceil\cdot\rceil$ for the least integer function.
\begin{enumerate}
\item\label{squarefreevanishing}
If $a>\lceil\frac{g}{3}\rceil$, then
\[
\mu_{1,g}'(a)
=\mu_{2,g}'(a)
=0.
\]
\item\label{topanumber}
If $g\equiv1\pmod{3}$ and $a=\lceil\frac{g}{3}\rceil$, then
\[
\mu_{1,g}'(a)
=q^{-2a+1}
\hspace{20px}\text{and}\hspace{20px}
\mu_{2,g}'(a)
=q^{-2a+1}\lp1+q^{-1}\rp.
\]
\end{enumerate}
\end{cor}

\cref{actualvalues} has the following application.

\begin{cor}\label{biganums}
Suppose $p=3$, $\epsilon\in\lb1,2\rb$, and $g,a,r$ are nonnegative integers.
\begin{enumerate}
\item
If $f\in\mathcal{P}'_{q,\epsilon}(g)$ and $a(f)=g-r$, then
\[
g\leq\frac{3r}{2}+1.
\]
\item
If $r\equiv0\pmod{2}$ and $g\equiv1\pmod{3}$, then there exists $f\in\mathcal{P}'_{q,\epsilon}(g)$ with $a(f)=g-r$.
\end{enumerate}
\end{cor}

More generally, \cref{bigone} implies bounds on probabilities for the $a$-numbers of squarefree polynomials for all $a\in\Z_{\geq0}$.

\begin{cor}\label{boundedvalues}
Suppose $p=3$, $\epsilon\in\lb1,2\rb$ and $g,a$ are nonnegative integers.
\begin{enumerate}
\item\label{corpart1}
Then
\[
\mu_{\epsilon,g}(a)\lp1+q^{-1}\rp-q^{-1}
\leq\mu_{\epsilon,g}'(a)
\leq\mu_{\epsilon,g}(a)\lp1+q^{-1}\rp.
\]
\item\label{corpart2}
If $a>0$, then
\[
\lv\mu_{\epsilon,g}'(a)-q^{-2a+1}\rv<2q^{-2a}.
\]
\end{enumerate}
\end{cor}

Previously the authors of~\cite{CEZB} introduced a heuristic prediction for the values of 
$\lim_{g\to\infty}\mu '_{1,g}(a)$ via a model 
of ``random $p$-divisible groups''.
Upon performing numerical computations they made the intriguing remark that the statistics of the $a$-numbers of random $p$-divisible 
groups appear to differ from those of hyperelliptic curves, especially when $p=3$.
Specifically they proved that the limiting probability of random $3$-divisible groups with $a$-number 0 
is $\prod_{i=1}^\infty (1+q^{-i})^{-1}$, while 
on the other hand their computations suggest that when $p=3$ the quantities 
$\mu'_{1,g}(a)$ might tend to $1-q^{-1}$ as $g$ increases.
Thus, one consequence of \cref{bigone} is that by broadening the set of  hyperelliptic curves considered, the corresponding probabilities 
suggested by the $p=3$ data of~\cite{CEZB} are true not just in the $g$ limit, but for every $g\geq3$.

In \cref{heuristicmodel} we introduce an alternative model for studying $a$-numbers of hyperelliptic curves (``random subspaces of fixed height'').
\cref{heurcomp} shows that the statistics of this model match that of the data computed by~\cite{CEZB}.
This theorem emboldens us to make \cref{wemadeaconj}, which predicts that for $p=3$, $\epsilon\in\lb1,2\rb$, and $a\in\Z_{\geq0}$,
\[
\lim_{g\to\infty}{\mu_{\epsilon,g} '(a)}=
\begin{cases}
1-q^{-1}&\text{if }a=0\\
q^{-2a+1}\lp1-q^{-2}\rp&\text{if }a>0.
\end{cases}
\]
This conjecture has the benefit of encompassing all $a\in\Z_{\geq0}$.
(When $a>0$, the data of \cite{CEZB} were not conclusive enough to suggest a limiting probability.)


Finally, in \cref{imoduli}, we apply these estimates of $\mu_{\epsilon,g} '(a)$ to compute the codimensions of the 
$a$-number strata of the moduli space of smooth hyperelliptic curves of genus $g$ over 
$\overline{\FF_q}$ in characteristic 3.
In particular, we prove in \cref{thm:strata} that when $a$ is positive these codimensions are $2a-1$.
By comparison, in \cite{O} the author showed that the codimensions of the $a$-number strata in the 
full moduli space of abelian varieties are $a(a+1)/2$.
Interestingly, this discrepancy contrasts $p$-rank strata of hyperelliptic curves; the results of \cite[Theorem~1]{GP} proves that the codimensions of the 
$p$-rank strata of hyperelliptic curves and those of abelian varieties are the same.

\subsection{Previous work}\label{prevv}

Many previous investigations into heuristics for hyperelliptic curves build upon work of 
Friedman and Washington~\cite{FW} who 
compute statistics of invertible matrices over finite fields and compare these statistics to those 
of the $\ell$-parts of Jacobians 
of hyperelliptic curves, where $\ell$ is a prime number satisfying $\ell\nmid2p$.
This endeavor is analogous to the Cohen-Lenstra heuristics~\cite{CL}, where the authors compute 
statistics of random finite abelian 
groups and compare them to those of ideal class groups of quadratic number fields.
Moreover, the heuristics of Friedman and Washington~\cite{FW} turn out to be provably correct 
as $q\to\infty$ with $q\not\equiv0\pmod\ell$, as shown in~\cite{EVW}.
In~\cite{G} the author addresses the case where $q\not\equiv 1\pmod{\ell}$, leaving open the case when $p = \ell$.
The heuristic investigation of this case begins with~\cite{CEZB}, as detailed above.

Other previous results on the $a$-numbers of hyperelliptic curves include the following.
Proposition~3.1 of~\cite{Re} is an upper bound on the $a$-numbers of any smooth, complete, irreducible 
curve in terms its genus.
Theorem~1.1 of~\cite{E} improves this theorem, restricting to Kummer covers of the projective line.
For hyperelliptic curves over finite fields of characteristic three, \hyperref[squarefreevanishing]{\cref*{actualvalues} (\ref*{squarefreevanishing})} improves upon both these bounds.
Moreover, \hyperref[topanumber]{\cref*{actualvalues}~(\ref*{topanumber})} implies that 
\hyperref[squarefreevanishing]{\cref*{actualvalues} (\ref*{squarefreevanishing})} is sharp when $g\equiv1\pmod{3}$ 
by giving explicit (positive) values of $\mu_{1,g} '(a)$ and $\mu_{2,g} '(a)$ in this case.

Another common means of addressing $a$-numbers is to ask: for which genera $g$ do there exist hyperelliptic curves 
of $a$-number $g-1$?
We mention the special cases of several results. If a hyperelliptic curve of genus $g$ over a finite field of 
characteristic three has $a$-number $g-1$, then:
\begin{itemize}
\item
\cite[Proposition~3.1]{Re} implies that $g\leq6$;
\item
\cite[Corollary~1.2]{E} implies that $g\leq4$; and
\item
\cite[Theorem~1.1]{Frei} implies that $g\leq2$.
\end{itemize}
As mentioned in the previous subsection, we show in \cref{biganums} that if $r\in\Z_{\geq0}$ and a hyperelliptic curve of genus $g$ over a finite 
field of characteristic three has $a$-number $g-r$, then
\[
g\leq\frac{3r}{2}+1.
\]
In particular, we recover Frei's result when $r=1$.
Moreover, when $r$ is even and $g\equiv1\pmod{3}$, we show that this bound cannot be improved.

Finally, \cite{Sankar} studies the statistics of $a$-numbers in the families of superelliptic curves 
in characteristic two and Artin-Schreier curves in any positive characteristic.
Interestingly, in odd characteristic the limiting probability that an Artin-Schreier curve has 
$a$-number 0 (as the genus of the curves grow) is zero, in contrast to what we conjecture to be the 
case for hyperelliptic curves in \cref{wemadeaconj}.

\subsection{Organization}\label{orgg}

This paper is organized as follows.
In \cref{heightheuristic} we review $a$-numbers and the Cartier operator; some of the
results in this section 
hold when $p$ is odd though we
restrict to the case $p=3$ to simplify the exposition.
In \cref{thunder} we use geometry of numbers for function fields to prove our fundamental height results; 
the results in this section hold for any prime $p$.
In \cref{consequences} we apply these results to deduce our main theorems. As noted, these results 
are only for the case $p=3$. 

\subsection{Competing interests declaration}\label{ideclare}

The authors declare no competing interests.

\section{Preliminary results} \label{heightheuristic}

In this section we will assume that $p=3$ (though much of it is true in the more general case where $p$ is
odd). Everything we do will revolve around particular 
polynomials attached to a given polynomial $f(X)\in\Fq [X]$.
Specifically, for any $f\in\Fq[X]$ let $c_0(f),c_1(f),$ and $c_2(f)$ be the unique polynomials in 
$\Fq[X]$ such that
$$
f(X)
=c_0(f)^3+c_1(f)^3\cdot X+c_2(f)^3\cdot X^2.
$$

In this section we first restrict to the case where $f$ is squarefree. We then recall the definition of the 
Cartier operator and see how the
$a$-number $a(H_f)$ may be computed via some linear algebra involving the coefficients $c_i(f)$. 
We then extend this to the case where $f$ is simply cubefree 
(this of course matches the previous when 
$f$ happens to be squarefree).
Finally, we recall relevant facts from the theory of heights in Diophantine geometry and show how 
computing the $a$-number may be reformulated in those terms.

In what follows we let $K$ be a fixed algebraic closure of $\Fq$.

\subsection{The Cartier operator and \textit{a}-numbers} \label{cartier}

Let $f\in\Fq[X]$ be a monic and squarefree polynomial. We denote the $K$-module of meromorphic differentials on $H_f$ defined over $K$ by
$\Omega _f^1$.

\begin{defn}\label{cartierdef}
For any monic squarefree polynomial $f\in\Fq[X]$, the \emph{Cartier operator} of $f$ is the (unique) map
\[
\mathcal{C}_f\colon\Omega_f^1\to\Omega_f^1
\]
that satisfies the following properties:
for any  $\omega,\omega_1,\omega_2\in\Omega^1_f$ and $z\in\mathcal{O}_{H_f}$,
\begin{itemize}
\item $\mathcal{C}_f\lp\omega_1+\omega_2\rp=\mathcal{C}_f\lp\omega_1\rp+\mathcal{C}_f\lp\omega_2\rp$;
\item$\mathcal{C}_f\lp z^p\omega\rp=z\cdot\mathcal{C}_f\lp\omega\rp$;
\item$\mathcal{C}_f\lp\der{z}\rp=0$; and
\item$\mathcal{C}_f\lp\frac{\der{z}}{z}\rp=\frac{\der{z}}{z}$.
\end{itemize}
\end{defn}

Set $\omega=\frac{\der{X}}{Y}\in\Omega^1_{f}$.
Then for any $j\in\Z_{\geq0}$ Elkin~\cite[Theorem~3.4]{E} proves that
\[ \numberthis\label{elkin1} 
\mathcal{C}_{f}\lp X^j\omega\rp =
\begin{cases}
X^jc_2(f)\omega &\text{if $j\equiv 0\mod 3$},\\
X^{j-1}c_1(f)\omega &\text{if $j\equiv 1\mod 3$},\\
X^{j-2}c_0(f)\omega &\text{if $j\equiv 2\mod 3$}.
\end{cases}
\]
Let $g$ denote the genus of the curve $H_f$, i.e., $\deg (f)=2g+\epsilon$ for $\epsilon = 1,\ 2$. Then by \cref{elkin1} 
we see that the Cartier operator yields a linear transformation of the $\Fq$-subspace of $\Omega _f^1$ spanned by $\{ \omega ,\ X\omega ,\ldots ,
X^{g-1}\omega\}$.
Section~5 of~\cite{O} shows that the $a$ number $a(H_f)$ is the dimension of the kernel of this transformation.

Now let us consider the case where
$f$ is possibly just cubefree. In this case we may uniquely write $f=f_1f_2^2$ where $f_1$ and $f_2$ are 
relatively prime and both are monic and squarefree.
(Of course, the case where $f_2=1$ is where $f$ is squarefree.)
Here
$f=
f_1f_2^3/f_2$. With this in mind,
an entirely similar proof as given by Elkin shows that \cref{elkin1} generalizes to
\[
\mathcal{C}_{f_1}\lp X^j\omega\rp =
\begin{cases}
X^jc_2(f)\omega &\text{if $j\equiv 0\mod 3$},\\
X^{j-1}c_1(f)\omega &\text{if $j\equiv 1\mod 3$},\\
X^{j-2}c_0(f)\omega &\text{if $j\equiv 2\mod 3$},
\end{cases}
\numberthis\label{elkin2}
\]
where now 
$\omega=\frac{\der{X}}{Yf_2}\in\Omega^1_{f_1}$. 
Set $d=\deg (f_2)$ and say $f=f_1f_2^2\in\cP _1(g)\cup\cP _2(g)$, so that the genus of $H_f$, which is the genus of $H_{f_1}$, is $g-d$. Of course 
$a(H_f)=a(H_{f_1})$, too. We will show the following.

\begin{lemma}\label{collins}
Let $f=f_1f_2^2\in \cP _1(g)\cup\cP _2(g)$ with $f_1$ and $f_2$ monic, squarefree, and relatively prime. 
Then $a(H_f)=0$ whenever $\deg (f)<3$ (i.e., when $g=0$).
If $\deg{f}\ge 3$ then \cref{elkin2} shows that $\mathcal{C}_{f_1}$ 
restricts to an action on the $\Fq$-subspace of $\Omega^1_{f_1}$ spanned by $\lb\omega,X\omega,\ldots,X^{g-1}\omega\rb$. The $a$-number $a(H_f)$ is 
the dimension of the kernel of this action.
\end{lemma}

\begin{example}
Let $f=X^3-X\in\mathcal{P}'_{1}(1)$.
Since
\[
f=(X)^3 + (-1)^3X
\]
we see that $c_0(f)=X$, $c_1(f)=-1$, and $c_2(f)=0$.
We set $\omega=\der{X}/Y$ and apply \cref{elkin1} to compute
\[
\mathcal{C}_f(\omega)= c_2(f)\omega = 0.
\]
By \cref{collins} we conclude that the $a$-number $a(H_f)=1$.
\end{example}

\begin{example}
Let
\begin{align*}
f&=f_1f_2^2=\lp X^3-X\rp\lp X^2+X+2\rp^2\\
&=X^7+2X^6+X^5+2X^4+2X^3+2X^2+2X\\
&=\lp 2X^2+2X\rp^3+\lp X^2+2X+2\rp^3X+(X+2)X^2,
\end{align*}
so that $f\in\mathcal{P}_{1}(3)$, and
\begin{align*}
c_0(f)&=2X^2+2X,\\
c_1(f)&=X^2+2X+2\text{, and }\\
c_2(f)&=X+2.
\end{align*}
As the squarefree part of $f$ is $f_1=X^3-X$ we set $\omega =\der{X}/Y\lp X^2+X+2\rp$ and use \cref{elkin2} to compute
\begin{align*}
\mathcal{C}_{f_1}\lp\omega\rp&=\lp X+2\rp\omega,\\
\mathcal{C}_{f_1}\lp X\omega\rp&=\lp X^2+2X+2\rp\omega\text{, and }\\
\mathcal{C}_{f_1}\lp X^2\omega\rp&=\lp2X^2+2X\rp\omega .
\end{align*}
Hence the image of $\lb\omega ,X\omega ,X^2\omega \rb$ under $\mathcal{C}_{f_1}$ has dimension 2, so that $a(H_f)=a(H_{f_1})=3-2=1$.
\end{example}

\begin{proof}
If $\deg{f}<3$, i.e., the genus of $H_f$ is zero, then the lemma is trivially true. We thus suppose $\deg{f}\ge 3$ so that $g\ge 1$.
If $d=0$ we are in the squarefree case, so that the lemma follows from the work of Oda in \cite{O} as remarked above.
Hence we may assume for the remainder that $d\ge 1$. 

Let $\alpha_1,\ldots,\alpha_d$ be the (distinct) roots of $f_2$ in $K$ and set $\omega=\frac{\der{X}}{Y}\in\Omega^1_{f_1}$ as above.
Via the partial fraction decomposition we see that
\[
\Span_K{\lp\lb\frac{\der{X}}{Yf_2(X)},
\ldots,\frac{X^{d-1}\der{X}}{Yf_2(X)}\rb\rp}
=\Span_K{\lp\lb\frac{\der{\lp X-\alpha_1\rp}}{Y\lp X-\alpha_1\rp},
\ldots,\frac{\der{\lp X-\alpha_d\rp}}{Y\lp X-\alpha_d\rp}\rb\rp}.
\]
For every $i\in\lb1,\ldots,d\rb$ choose $\beta_i\in K$ with $\beta_i^2=f_1\lp\alpha_i\rp$, and let $P_i$ be the point $\lp\alpha_i,\beta_i\rp$ on $H_{f_1}$.
With this notation, we see that for any such $i$,
$\displaystyle{\ord_{P_i}{\lp\frac{\der{\lp X-\alpha_i\rp}}{\lp X-\alpha_i\rp}\rp}=-1}$. Also, 
if $j\in\lb1,\ldots,d\rb\setminus{\lb i\rb}$ then $\displaystyle{\ord_{P_j}{\lp\frac{\der{\lp X-\alpha_i\rp}}{\lp X-\alpha_i\rp}\rp}=0.}$ Finally,
$\ord_{P_i}{\lp\frac{1}{Y}\rp}=0$ since $\gcd{\lp f_1,f_2\rp}=1$.

We now recall
(see Exercise~4.14 of~\cite{Stich}, for example) that if $\omega$ is holomorphic, 
then so is $\mathcal{C}_f(\omega ),$ and that 
if $P$ is a point on $H_f$ and $\ord_P{(\omega)}=-1$, then $\ord_P{\lp\mathcal{C}_f(\omega)\rp}=-1$.
Thus $\mathcal{C}_{f_1}$ acts injectively on $\Span_K{\lp\lb\frac{\der{X}}{Yf_2(X)},
\ldots,\frac{X^{d-1}\der{X}}{Yf_2(X)}\rb\rp}$.
In particular, if $d=g$ then the genus of $H_f$ is zero and so is the $a$-number $a(H_f)$, which is also the dimension of the kernel of our Cartier operator 
linear transformation. 
Thus we now suppose that $d<g$.

Via long division and the partial fraction decomposition we see that
\begin{align*}
&\Span_K{\lp\lb\frac{\der{X}}{Yf_2(X)},
\ldots,\frac{X^{g-1}\der{X}}{Yf_2(X)}\rb\rp}\\
&\hspace{20px}=\Span_K{\lp\lb\frac{\der{X}}{Yf_2(X)},
\ldots,\frac{X^{d-1}\der{X}}{Yf_2(X)}\rb
\bigcup\lb\frac{\der{X}}{Y},
\ldots,\frac{X^{g-d-1}\der{X}}{Y}\rb\rp}\\
&\hspace{20px}=\Span_K{\lp\lb\frac{\der{\lp X-\alpha_1\rp}}{Y\lp X-\alpha_1\rp},
\ldots,\frac{\der{\lp X-\alpha_d\rp}}{Y\lp X-\alpha_d\rp}\rb
\bigcup\lb\omega,X\omega,
\ldots,X^{g-d-1}\omega\rb\rp}.
\end{align*}
As the genus of $H_{f_1}$ is $g-d$, we know that $\omega,X\omega,
\ldots,X^{g-d-1}\omega$ are holomorphic.
Thus, it follows that
\[
\lb0\rb=\Span_K{\lp\lb\frac{\der{\lp X-\alpha_1\rp}}{Y\lp X-\alpha_1\rp},
\ldots,\frac{\der{\lp X-\alpha_d\rp}}{Y\lp X-\alpha_d\rp}\rb\rp}
\bigcap
\Span_K{\lp\lb\omega,X\omega,
\ldots,X^{g-d-1}\omega\rb\rp},
\]
finishing the proof.
\end{proof}

\begin{lemma}\label{Apolys} 
Let $f=f_1f_2^2\in\Fq[X]$ be as above and let
$$Q(X)\omega=\sum _{i=0}^{g-1}a_iX^i\omega$$
be an element of the $\Fq$-subspace of $\Omega _{f_1}^1$ in \cref{collins}.
Then $Q(X)\omega$ is in the kernel of $C_{f_1}$ if and only if
\[
\numberthis\label{fundeq}
c_2(Q)c_0(f)+c_1(Q)c_1(f)+c_0(Q)c_2(f)=0.
\]
\end{lemma}

\begin{proof}
Via \cref{elkin2} we readily verify the following:
$$C_{f_1}(X^{3j}\omega )=c_2(f)X^j\omega ,\qquad
C_{f_1}(X^{3j+1}\omega )=c_1(f)X^j\omega ,\qquad
C_{f_1}(X^{3j+2}\omega )=c_0(f)X^j\omega .$$
The lemma follows.
\end{proof}

Combining the two lemmas above yields the main result of this subsection. (We adopt the usual convention here that the
degree of the zero polynomial is $-\infty$.)

\begin{prop}\label{subsum} Let $p=3$ and $g\in\Z _{\ge 0}$. Then for all $f\in\cP _1(g)\cup\cP _2(g)$ 
the $a$-number $a(H_f)$ is the dimension
of the $\Fq$-vector space consisting of those $Q(X)\omega$ where $Q$ satisfies \cref{fundeq}
and 
\[
\deg \big ( c_2(Q)\big )\le \frac{g-3}{3},\qquad
\deg \big ( c_1(Q)\big )\le \frac{g-2}{3},\qquad
\deg \big ( c_0(Q)\big )\le \frac{g-1}{3}.
\]
\end{prop}

\subsection{Heights and \textit{a}-numbers} \label{heuristic}

To ease notation we will write $F$ for the field of rational functions $\Fq (X)$ in this subsection.
We first recall some more elementary notions for the field $F$. Of course height machinery encompasses 
a much wider range of 
``global fields," but  here we will limit ourselves to the needs at hand.

It is well known that the places of $F$ are in one-to-one correspondence with the set of monic 
irreducible polynomials $P\in \Fq [X]$ 
together with the ``place at infinity" $v_0$ corresponding to the usual degree function. 
The degree of the place $v_0$ is one; the degree of
the place corresponding to the monic irreducible $P$ is simply the degree of $P$ as a polynomial. 
For each place $v$ we have the usual notion of the order $\ord _v$ (the order at $v_0$ of
a rational function being the negative of its degree, the order at the place corresponding to $P$ being the 
exact power of $P$ dividing the 
rational function). We will write
$\Div (F)$ for the divisor group - the free abelian group generated by the set of places of $F$. 
For any non-zero element $\alpha\in F$ we
get a well-defined principal divisor $\divi (\alpha )=\sum _{v}\ord _v(\alpha )\cdot v$. 
The ``product formula" in this particular case states that
the degree of any principal divisor is zero:
$$\deg \big (\divi (\alpha )\big ) =\sum _{v}\ord _v(\alpha)\deg (v)=0.$$
This is readily verified in our case here.

We extend the notion of order to $n$-tuples as follows:
$$\ord _v(\alpha _1,\ldots ,\alpha _n)=\min \{ \ord _v(\alpha _1),\ldots ,\ord _v(\alpha _n)\}.$$
This in turn yields a well-defined divisor (no longer necessarily principal!):
$$\divi (\ua )=\sum _v\ord _v(\ua )\cdot v.$$
Via this we have the definition of the usual height of a non-zero vector $\mathbf{\alpha}\in F^n$:
$$h(\ua ) =-\deg\big ( \divi (\ua )\big ).$$
Since the degree of a principal divisor is zero, we immediately see that this height function $h$ is 
really a function on one-dimensional subspaces of $F^n$. 

All the above can be extended to arbitrary subspaces, but for us here we only need the fact that 
the height $h(V)$ of an $(n-1)$-dimensional hyperplane $V\subset F^n$
is the height of its orthogonal complement. In other words, if $V$ is the subspace of $\ub$ such 
that $\ua\cdot\ub =0$, then $h(V)=h(\ua )$. We also set 
the height of the full space $F^n$ and its orthogonal complement $\{\uo \}$ to both be zero.

\begin{example}\label{heightexample} 
Via an appropriate scalar multiplication, any non-zero $\ua\in F^n$ is projectively equivalent to a non-zero $n$-tuple
of polynomials $(A_1,\ldots ,A_n)$ whose greatest common divisor is 1. For such a tuple, we clearly have 
$\ord _P(A_1,\ldots ,A_n)=0$
for any place $P$ where $P$ is a monic irreducible polynomial. For the place at infinity $v_0$ we have
$$\ord _{v_0}(A_1,\ldots ,A_n)=\min \{-\deg (A_1),\ldots ,-\deg (A_n)\} .$$
Thus
$$h(A_1,\ldots ,A_n)=\max \{\deg (A_1),\ldots ,\deg (A_n)\}.$$
Moreover, the height of the orthogonal complement
$\{ \ub : \ub\cdot\ua =0\}$ is also the maximum of the degrees of the $A_i$'s.
\end{example}

Given the example above, the following lemma indicates why one is lead to heights when contemplating the 
$a$-number $a(H_f)$. (We note for posterity
that the obvious generalization of this lemma is valid for all odd primes $p$, not just $p=3$.)

\begin{lemma}\label{gcd_cf} 
For all $f_1, f_2 \in \mathbb{F}_q[X]$, $c_i(g)=f_2c_i(f_1)$ for $i=0,1,2$ if and only if $g=f_1f_2^3$.
In particular, $f$ is cubefree
if and only if the coefficients $c_0(f),\ c_1(f),\ c_2(f)$ are relatively prime and in that case the height
$$h\big ( c_0(f),c_1(f),c_2(f)\big ) =\max\{\deg\big (c_0(f)\big ) ,\deg\big ( c_1(f)\big ) ,
\deg\big ( c_2(f)\big )\} .$$ 
\end{lemma}

\begin{proof}
Clearly it suffices to prove this in the case where $f_2$ is just a monomial: $f_2=X^j$ for some $j>0$. 
One readily verifies that for any $f(X)=a_0+a_1X +\cdots\in\mathbb{F}_q[X]$
$$\begin{aligned}
c_0(f)&=\sum _{i\equiv 0\mod 3}a_iX^{\frac{i}{3}}\\
c_1(f)&=\sum _{i\equiv 1\mod 3}a_iX^{\frac{i-1}{3}}\\
c_2(f)&=\sum _{i\equiv 2\mod 3}a_iX^{\frac{i-2}{3}}.\end{aligned}
$$
The case where $f_2=X^j$ follows directly from this.
\end{proof}

Given any subspace $V\subseteq F^n$ there is the notion of {\it successive minima}.
In the case of a two-dimensional space $V\subset F^3$ there are two minima, $\mu _1(V),\ \mu _2(V)$, defined as
\begin{align*}
\mu _1(V)
&=\min\{ h(\ux _1)\}\mid\ux _1\in V\setminus\{\mathbf{0}\}\}\text{ and}\\
\mu _2(V)
&=\min\{\max\lb h(\ux _1),h(\ux _2)\rb
\mid\lb\ux _1,\ux _2\rb\text{ is a basis for }V\}.
\end{align*}
We will make frequent use use of Minkowski's Theorem (\cite[Theorem 1]{T}), which in this instance reads
\[
\mu _1(V)\le\mu _2(V),\quad \mu _1(V)+\mu _2(V)=h(V).\numberthis\label{mink}
\]

\begin{thm}\label{anumberheights} Set $p=3$, let $g\ge 0$ and $\epsilon\in\{ 1,2\}$. Set $m=[(g-1)/3]$ and 
suppose $f=f_1f_2^2\in\cP _\epsilon (g)$ with $f_1$ and $f_2$ relatively prime.

If $g=0$ then $a(H_f)=0$. 

Suppose $g\ge 1$ and set $V\subset F^3$ to be the $2$-dimensional
orthogonal complement of $\big ( c_0(f),c_1(f),c_2(f)\big )$ so that
\[ h(V)=
\max\{\deg \big ( c_0(f)\big ), \deg \big ( c_1(f)\big ) ,\deg\big ( c_2(f)\big )\}
.\]
Then 
\[ h(V)=
\begin{cases} 
\frac{2g+\epsilon}{3}=
\deg\big ( c_0(f)\big )>\max\{\deg\big ( c_1(f)\big ) ,\deg\big ( c_2(f)\big )\}
&\text{if $2g+\epsilon\equiv 0\pmod 3$,}\\
\frac{2g+\epsilon -1}{3}
=\deg\big ( c_1(f)\big ) >\deg\big ( c_2(f)\big ) 
&\text{if $2g+\epsilon\equiv 1\pmod 3$,}\\
\frac{2g+\epsilon -2}{3}
=\deg\big ( c_2(f)\big ) 
&\text{if $2g+\epsilon\equiv 2\pmod 3$,}
\end{cases}
\]
and $a(H_f)>0$ only if $\mu _1(V)\le m.$ 

Suppose $Q\in\mathbb{F}_q[X]$ is cube-free and satisfies \cref{fundeq}, i.e., 
$\big ( c_2(Q), c_1(Q), c_0(Q)\big )\in V$ and $c_0(Q),c_1(Q), c_2(Q)$ are relatively prime. Suppose further
that
\[
\mu _1(V)=h\big ( c_2(Q),c_1(Q),c_0(Q)\big )\le m
.\]
Unless we have one of the following three
situations:
\begin{enumerate}
\item
$g\equiv 1\pmod 3$, $\epsilon = 2$, $\deg \big ( c_2(Q)\big ) =\mu _1(V);$
\item
$g\equiv 1\pmod 3$, $\epsilon = 1$, $\deg \big ( c_1(Q)\big ) =\mu _1(V);$
\item 
$g\equiv 2\pmod 3$, $\epsilon = 1$, $\deg \big ( c_2(Q)\big ) =\mu _1(V);$
\end{enumerate}
then the kernel of $C_{f_1}$ is spanned by 
$$\{ X^{3j}Q\omega : 0\le j\le m-\mu _1(V)\}$$  
so that the $a$-number 
$$a(H_f)=m -\mu _1(V) +1.$$
In any of (1), (2), or (3) hold then the kernel of $C_{f_1}$ is spanned by
$$\{ X^{3j}Q\omega : 0\le j\le m-\mu _1(V)-1\}$$ 
if $m>\mu _1(V)$ or simply $\{ 0\}$ if $m=\mu _1(V)$,
so that the $a$-number 
$$a(H_f)=m -\mu _1(V).$$
\end{thm}

\begin{proof}
The first part of the theorem follows directly from \cref{gcd_cf} and the definitions.

For the next part, \cref{subsum} and \cref{heightexample} immediately implies that $a(H_f)>0$ only if 
$\mu _1(V)\le\frac{g-1}{3}$.

Now suppose $\mu _1(V)=h\big ( c_2(Q),c_1(Q),c_0(Q)\big )\le\frac{g-1}{3}$ where $Q$ satisfies \cref{fundeq}. 
To ease notation temporarily write $l$ for $\mu _1(V)$.
There are only three possible scenarios where $X^{3i}Q$ don't satisfy the degree requirements in \cref{subsum}
for all $0\le i\le m-l$: 
$\deg \big ( X^{m-l}c_2(Q)\big ) =\frac{g-1}{3}$ (and so $\deg \big ( c_2(Q)\big ) =l$ and
$g\equiv 1\pmod 3$),
$\deg \big ( X^{m-l}c_2(Q)\big ) =\frac{g-2}{3}$ (and so $\deg \big ( c_2(Q)\big ) =l$ and 
$g\equiv 2\pmod 3$), 
and $\deg \big ( X^{m-l}c_1(Q)\big )=\frac{g-1}{3}$ 
(and so $\deg \big ( c_1(Q)\big ) =l$ and $g\equiv 1\pmod 3$ once again). 
Note that in all three cases here we {\it do} have $X^{3i}Q$ satisfies the degree requirements 
in \cref{subsum}
for all $0\le i< m-l$ assuming $m-l>0$; if $m=l$ then $Q$ does not satisfy the
degree requirements in \cref{subsum}.
In that first
scenario, we can't have $\deg \big ( c_0(f)\big )$ strictly larger than $\deg\big ( c_1(f)\big )$ and 
$\deg\big  (c_2(f)\big )$, 
since $\deg \big ( c_0(Q)\big )$ is at least as large as
$\deg\big ( c_1(Q)\big )$ and $\deg\big ( c_0(Q)\big )$ by hypothesis and $Q$ satisfies \cref{fundeq}. 
By what we have already shown this implies that
$2g+\epsilon\not\equiv 0\pmod 3$, so that $\epsilon \neq 1$. The second and third scenarios are dealt with 
in an entirely similar manner.

If we don't have one of the situations (1), (2), or (3) and $\mu _1(V)\le m$, 
then we have shown that $Q$
satisfies the degree requirements in \cref{subsum} and thus $Q\omega$ 
is in the kernel of the 
Cartier operator. Moreover, if $\big ( c_2(P), c_1(P), c_0(P)\big )$
is any other element of $V$ linearly independent from $\big ( c_2(Q),c_1(Q),c_0(Q)\big )$, 
then its height is at least 
$\mu _2(V)$. By \cref{mink} and our
calculation of $h(V)$, this implies that $\max\{\deg \big ( c_2(P)\big ),\deg\big  ( c_1(P)\big ),
\deg\big ( c_0(P)\big )\}>\frac{g-1}{3}$, 
thus $P\omega$ is not in the kernel. This shows that all $P\in\mathbb{F}_q[X]$ with $P\omega$ in
the kernel of the Cartier operator are $\mathbb{F}_q[X]$-multiples of $Q$. The theorem now follows
from \cref{gcd_cf} and \cref{subsum}.
\end{proof}

\section{Heights of subspaces}\label{thunder}

The results from this section hold for finite fields of arbitrary characteristic, so suppose that $q$ is any prime power.
We continue to write $F=\Fq (X)$, where $X$ is transcendental over $\Fq$.
In this section, we compute statistics of subspaces of $F^3$ with specified height.

It will prove useful to allow local changes of coordinates. Specifically, if $\cA = (\cA _v)\in\gl _n(F_{\A})$ 
(the indexing here is over the places $v$ of $F$) is
an element of the general linear group over the adele ring $F_{\A}$, then for any non-zero $\ua\in F^n$ 
we set the ``twisted height" $h_{\cA}(\ua )$ to be
$$h_{\cA}(\ua )=-\deg \left ( \sum _v\ord _v\big (\cA _v(\ua )\big ) \cdot v\right ).$$
(Since $\cA\in\gl _n(F_{\A})$ the local changes of coordinates only change orders at finitely many places, 
so that this is well-defined.)
In this case we also set 
$$h_{\cA}(\{\uo \} )=0,\qquad h_{\cA }(F^n)=-\deg\left (\sum _v \det (\cA _v)\cdot v\right  ) .$$
Obviously the usual height $h$ without the subscript is simply the case where $\cA$ is the identity. 

\subsection{Intermediate counting results}\label{intermediate}

Our first result uses height machinery to count the number of 2-dimensional subspaces of $F^3$ 
containing a given non-zero vector.
We begin with some notation.
For any $n\in\Z_{\geq1}$ and $k\in\Z_{\geq0}$, let
$N_{F^n}(k)$ denote the number of one-dimensional subspaces of $F^n$ of (untwisted) height $k$.
Further, if $W\subset F^3$ is a one-dimensional subspace of $F^3$, then we let
$N_{F^n/W}(k)$ denote the number of $2$-dimensional subspaces $V\supset W$ with height $h(V)=h(W)+k$.

It turns out that
\[
N_{F^n}(k)=
\begin{cases} q^{nk-n+1}(q^n-1)(q^{n-1}-1)/(q-1)&\text{if $k\ge 1$,}\\
(q^n-1)/(q-1)&\text{if $k=0$.}\end{cases}\numberthis\label{2}
\]
Indeed, this follows from the proof of the following proposition.

\begin{prop}\label{lkheight}
Suppose that $l\in\Z_{\geq0}$, that $k\in\Z_{\geq1}$, and that $W\subset F^3$ is a subspace of height $l$.
Then
\[
N_{F^3/W}(k)=q^{2k+l-1}(q^2-1).
\]
\end{prop}

\begin{proof}
As noted above, our method of proof applies to \cref{2} as well; we will make appropriate remarks in
this regard throughout the proof.

Let $\bfa$ be the adele ring of $F$. Then according to \cite[Lemma 8b]{T} there is an $\cA\in\gl_2(\bfa )$ 
giving a ``twisted height'' $h_{\cA}$ on the two-dimensional factor space $F^3/W$ such that 
$h_{\cA}(V/W)=h(V)-h(W)$ 
for all subspaces $V\supseteq W$.
Thus $N_{F^3/W}(k)$ is the number of one-dimensional subspaces of $F^3/W$ with twisted height $k$. 
Similar to \cref{mink} above we have Minkowski's Theorem:
\begin{align*}
\mu _1(F^3/W)+\mu _2(F^3/W)&=h_{\cA}(F^3/W)=h(F^3)-h(W)=-l.\numberthis\label{4}
\end{align*}
We see that $\mu _1(F^3/W)\ge 0-h(W)=-l$ since $h(V)\ge 0$
for all subspaces $V\subseteq F^3$.
Hence
\[
-l\le\mu _1(F^3/W)\le -l/2,\qquad \mu _1(F^3/W)\le\mu _2(F^3/W)\le 0.\numberthis\label{5}
\]
Note that in the case $l=0$ we get $\mu _1(F^3/W)=\mu _2(F^3/W)=0$.
We will use \cref{5} together with techniques from \cite{T} to evaluate the counting function $N_{F^3/W}$.

Our counting function can be expressed in terms of the dimensions of certain Riemann-Roch spaces and divisors.
All necessary background on this may be found in \cite{TSiegel,T}. 
For any divisor $\fa\in\Div (F)$, the divisor group of $F$, and any $\cA\in\gl _n{(\bfa)}$ set
\[
L(\fa ,\cA)=\{\ux\in F^n\mid
\ord_v{(\cA _v\ux )}\ge -\ord_v{(\fa)}\ 
\text{for all $v\in M(F)$}\},
\]
where $M(F)$ denotes the set of places of the field $F$.
As shown in \cite[\S  II]{TSiegel}, the set $L(\fa ,\cA)$ is a finite-dimensional vector space over $\Fq$;
denote its dimension by $l(\fa ,\cA )$. Now suppose $\cA$ arises from a factor space $F^3/W$ as above.
Since $\mu_2(F^3/W)\le 0$ by \cref{5}, for any divisor $\fa$ with $\deg{(\fa )}\ge -1$ we have
\[
l(\fa ,\cA)
=2\big (\deg (\fa )+1\big )
+\deg\divi\det (\cA )
=2\big (\deg (\fa )+1\big )+l\numberthis\label{6}
\]
by \cite[Lemma 10]{T}.
Note that the height of the full space $h(F^3/W)=-\deg\divi\det (\cA)$ always.
Similarly, since $\mu _1(F^n)=\mu _2(F^n)=\cdots =\mu _n(F^n)=0=\deg\divi\det (I_n)$ 
for the identity element $I_n\in\gl _n (F_\A)$,
\[
l(\fa ,I_n)=n\big (\deg (\fa )+1\big )\refstepcounter{primes}\tag{\theequation'}\label{6'}
\]
for all divisors $\fa$ with $\deg (\fa )\ge -1$.

Fix a divisor $\fa _0$ of degree 1. Then for any $j\ge -1$
\[
l(j\fa _0,\cA)=2(j+1)+l,\numberthis\label{7}
\]
\[
l(j\fa _0,I_n)=n(j+1)\refstepcounter{primes}\tag{\theequation'}\label{7'}
\]
by \cref{4}, \cref{5}, \cref{6}, and \cref{6'}.

Next, for any integer $j\ge 0$ set
\[
b(j)=\sum_{\substack{\fc\ge 0\\ \deg (\fc )=j}}{\mu (\fc )},
\]
where here $\mu$ denotes the usual M\"obius function on effective divisors.
By \cite[Lemma 12]{T}
\[
b(j)=\begin{cases}
1&\text{if $j=0$,}\\
-(q+1)&\text{if $j=1$,}\\
q&\text{if $j=2$,}\\
0&\text{if $j>2$.}\end{cases}\numberthis\label{8}
\]

With the above notation in mind, \cite[(5)]{T} reads
\[
(q-1)N_{F^3/W}(k)
=\sum _{j=0}^{k-\mu _1(F^3/W)}b(j)\big ( q^{l((k-j)\fa _0,\cA)}-1\big ).
\]

Note by \cref{4} that $k-\mu _1(F^3/W)\ge 2$ if $l,k\ge 1$ or if $l=0$ and $k\ge 2$.
Thus the former equation above, in conjunction with \cref{6}, \cref{7}, and \cref{8}, yields
\begin{align*}
(q-1)N_{F^3/W}(k)&=\big (q^{l(k\fa _0,\cA)}-1\big ) -(q+1)\big (q^{l((k-1)\fa _0,\cA)}
-1\big ) +q\big (q^{l((k-2)\fa _0,\cA)}-1\big )\\
&=\big (q^{2(k+1)+l}-1\big ) -(q+1)\big ( q^{2k+l}-1\big )
+q\big ( q^{2(k-1)+l}-1\big )\\
&= q^{2k+l}\big (q^2-(q+1)+q^{-1}\big )\\
&=q^{2k+l-1}(q-1)(q^2-1)
\end{align*}
provided $l,k\ge 1$ or $l=0$ and $k\ge 2$.
In the case $l=0$ we have $\mu _1(F^3/W)=0$ as noted above, so that when $k=1$ we get
\begin{align*}
(q-1)N_{F^3/W}(1)&=\big (q^{l(\fa _0,\cA)}-1\big ) -(q+1)\big (q^{l(0,\cA)}-1\big )\\
&=(q^4-1)-(q+1)(q^2-1)\\
&= q^4-q^3-q^2+q\\
&=(q-1)(q^3-q).
\end{align*}

Similarly, since $\mu _1(F^n)=0$, we see that~\cite[(5)]{T} reads
\[
(q-1)N_{F^n}(k)=\sum _{j=0}^{k}b(j)\big ( q^{l((k-j)\fa _0, I_n)}-1\big).
\]
This together with \cref{6'}, \cref{7'}, and \cref{8} yield
\begin{align*}
(q-1)N_{F^n}(k)&=(q^{n(k+1)}-1)-(q+1)(q^{nk} -1)+q(q^{n(k-1)}-1)\\
&=q^{nk+n}-q^{nk+1}-q^{nk}+q^{nk-n+1}\\
&=q^{nk}(q^n-1)-q^{nk-n+1}(q^n-1)\\
&=q^{nk-n+1}(q^n-1)(q^{n-1}-1)
\end{align*}
if $k\ge 2$,
\begin{align*}
(q-1)N_{F^n}(1)&=(q^{n2}-1)-(q+1)(q^{n} -1)\\
&=q^{n2}-q^{n+1}-q^n+q\\
&=q^n(q^n-1)-q(q^n-1)\\
&=q(q^n-1)(q^{n-1}-1),
\end{align*}
and finally
\[
(q-1)N_{F^n}(0)=q^n-1.
\]
\end{proof}

In what follows it will prove convenient to fix some notation for the various sets we will count.

\begin{defn}\label{Sdef}
For any positive integers $m$ and $l$ let $S(m)$ denote the set of all ordered triples $(A_1,A_2,A_3)$ 
of relatively prime polynomials in $\Fq [X]$ with maximal degree $m$ and $S(m,l)$ denote the subset of 
$S(m)$ consisting of those triples such that there is a solution $\ux = (x_1,x_2,x_3)\in F^3$ of height $l$ to the 
homogeneous linear equation 
\[ x_1A_1+x_2A_2+x_3A_3=0.\numberthis\label{1}
\]
\end{defn}

As discussed previously, the set $S(m,l)$ corresponds to the two-dimensional subspaces 
$V\subset F^3$ with $h(V)=m$ that contain a one-dimensional subspace $W\subset V$ with $h(W)=l$.
In fact each such subspace $V$ corresponds to precisely $q-1$ triples in $S(m,l)$; this takes into account 
multiplication by elements of $\lp\Fq\rp^{\times}$.
The required solution(s) to \cref{1} of height $l$ are also ordered triples of relatively prime 
polynomials $(P_1,P_2,P_3)$ after multiplication by a suitable non-zero element of $F$. 
We note by Minkowski's Theorem \cref{mink} that this
solution $(P_1,P_2,P_3)$ is projectively unique if $m>2l$.

\begin{defn}\label{Tdef}
For any positive integers $m$ and $l$
let $T(m,l)$ denote the subset of $S(m,l)$ where there is an ordered triple of relatively prime
polynomials $(P_1,P_2,P_3)$ that is a solution to \cref{1} 
of height $l$ with the additional stipulation that $\deg (P_2)=l$. 
Let $T'(m,l)$ denote the subset of $S(m,l)$ where there is a solution to \cref{1} of height $l$ with
$\deg (P_3)=l>\max\{\deg (P_1),\deg (P_2)\}$.
\end{defn}

We will also need to further delineate these ordered triples above according to which entry has the largest
degree.

\begin{defn}\label{Rdef}
For any non-empty subset $R\subseteq
\{1,2,3\}$ let $S_R(m),\  S_R(m,l)$, $T_R(m,l)$, and $T'(m,l)$ denote the subset of $S(m),\ S(m,l)$, $T(m,l)$, and
$T'(m,l)$,
respectively,
consisting of those ordered
triples $(A_1,A_2,A_3)$ where the degree of $A_i$ is $m$ if and only if $i\in R$. 
\end{defn}

As a notational
convenience we will drop the set brackets in the subscript here, e.g., we will write $S_{1,3}(m,l)$
instead of the more cumbersome $S_{\{ 1,3\} }(m,l)$.

We have determined the connection between curves with given $a$-numbers and these sets above in 
\cref{anumberheights}. 
Our goal for the remainder of this subsection is to determine the cardinalities of these sets above. 

\begin{lemma}\label{Thunder1}
The sets $S(m)$,  $S(m,l)$, $T(m,l)$, and $T'(m,l)$ are disjoint unions: 
\begin{align*}
S(m)&=S_1(m)\cup S_2(m)\cup\cdots \cup S_{1,2,3}(m),\\
S(m,l)&=S_1(m,l)\cup S_2(m,l)\cup\cdots \cup S_{1,2,3}(m,l)\\
T(m,l)&=T_1(m,l)\cup T_3(m,l)\cup T_{1,3}(m,l)\cup T_{1,2}(m,l)\cup T_{2,3}(m,l)\cup T_{1,2,3}(m,l)\text{, and}\\
T'(m,l)&=T'_1(m,l)\cup T'_2(m,l)\cup T'_{1,2}(m,l)
.\end{align*}
Further, 
\begin{gather*}
|S_1(m)|=|S_2(m)|=|S_3(m)|, \qquad |S_{1,2}(m)|=|S_{1,3}(m)|=|S_{2,3}(m)|,\\
|T_1(m,l)|=|T_3(m,l)|, \qquad  |T_{1,2}(m,l)|=|T_{2,3}(m,l)|,\\
|T_1'(m,l)|=|T_2'(m,l)|,
\end{gather*}
and if $m>2l$
\[
T_{1,2}'(m,l)= S_{1,2}(m,l)\setminus T_{1,2}(m,l).
\]
\end{lemma}

\begin{proof}
That $S(m)$ and $S(m,l)$ are the given disjoint unions is obvious from the definitions. The
same is true for $T(m,l)$ except here $T_2(m,l)$ is empty. To see why,
we note that if $(A_1,A_2,A_3)\in T(m,l)$ with $\deg (A_2)>\max\{\deg (A_1),\deg (A_3\})$ and 
$(P_1,P_2,P_3)$ is any triple of polynomials with $\deg (P_2)=\max\{\deg (P_1),\deg (P_2),\deg (P_3)\}$, then
the degree of $A_1P_1+A_2P_2+A_3P_3$ is $\deg (A_2)+\deg (P_2)\ge 0$, so that
$(P_1,P_2,P_3)$ is not a solution to \cref{1}. A similar argument shows that $T'_R(m,l)$ is empty whenever $3\in R$. 

Next, any element $\sigma$ of the symmetric group on $\{ 1,2,3\}$
gives rise to a permutation of $S(m)$ and $S(m,l)$ via
$(A_1,A_2,A_3)\mapsto (A_{\sigma (1)},A_{\sigma (2)},A_{\sigma (3)})$.
One readily checks that this mapping takes $S_i(m)$ to $S_{\sigma (i)}(m)$ and $S_{i,j}(m)$ to
$S_{\sigma (i),\sigma (j)}(m)$
for any $i,j\in\{1,2,3\}$ in a one-to-one and onto manner.
Further, the permutation that permutes $1$ with $3$ and leaves $2$ fixed takes $T_1(m,l)$
to $T_3(m,l)$ and
$T_{1,2}(m,l)$ to $T_{2,3}(m,l)$. The permutation that permutes $1$ with $2$ and leaves $3$ fixed takes
$T'_1(m,l)$ to $T'_2(m,l)$.  

Finally, suppose $(A_1,A_2,A_3)\in S_{1,2}(m,l)\setminus T_{1,2}(m,l)$ and $m>2l$. Then by definition 
$\deg (A_1)=\deg (A_2)
=m>\deg (A_3)$ and as noted above the projectively unique solution $(P_1,P_2,P_3)$ to \cref{1} 
satisfies $\deg (P_2)<l=\max\{ \deg (P_1),\deg (P_3)\}.$
This can only be the case if $\deg (P_3)=l>\deg (P_1)$ since otherwise we would have
\[ \deg (P_1A_1) = l+m>\max\{\deg (P_2A_2),\deg (P_3A_3)\},
\] 
contradicting the assumption that
$(P_1,P_2,P_3)$ is a solution to \cref{1}. This completes the proof.
\end{proof}

Consider the functions $\mapf,\mapg\colon\Fq [X]^3\times\lp\Fq\rp^{\times}\rightarrow \Fq [X]^3$ given by
\[
\mapf\big ( (Q_1,Q_2,Q_3),c\big ) = (Q_1,Q_2,Q_3+cQ_1)
\]
and
\[
\mapg\big ( (Q_1,Q_2,Q_3),c\big ) = (Q_1-cQ_3,Q_2,Q_3).
\]

\begin{lemma}\label{SolTranslate}
The function $\mapf$ is one-to-one when restricted to those triples where $Q_1\neq 0$. 
Both $\mapf$ and $\mapg$ take relatively prime triples to 
relatively prime triples, and neither $\mapf$ nor $\mapg$ alter the height of $(Q_1,Q_2,Q_3)$, as long as 
\[
\deg{(Q_1)}\ge\max\{\deg{(Q_2)},\deg{(Q_3)}\}
\]
or
\[
\deg{(Q_1)}\le \max\{\deg{(Q_2)},\deg{(Q_3)}\},
\]
respectively.
\end{lemma}

\begin{proof}
Indeed,
\[
\mapf\big ( \mapf\big ( (Q_1,Q_2,Q_3),c \big ),-c\big )=(Q_1,Q_2,Q_3)
\]
for all $Q_1,Q_2,Q_3\in \Fq [X]$, and 
\[
\mapf\big ( ( Q_1,Q_2,Q_3),c\big ) =\mapf\big ( (Q_1,Q_2,Q_3),b\big )
\]
with $Q_1\neq 0$ implies that $c=b$. An entirely similar argument works for $\mapg$. 
For the second part,
the common divisors of two polynomials $A$ and $B$ are the common divisors of $A$ and $B+CA$ for all polynomials $C$. 
The statement regarding heights now follows since neither $\mapf$ nor $\mapg$ alter the maximum degree under the two 
respective conditions on the degree of $Q_1$. 
\end{proof}

\begin{lemma}\label{precounting}
We have 
\begin{gather*}
|S_{1,3}(m)|=(q-1)|S_1(m)|,\qquad |S_{1,2,3}(m)|=(q-1)|S_{1,2}(m)|,\\
|T_{1,3}(m,l)|=(q-1)|T_1(m,l)|,\qquad |T_{1,2,3}(m,l)|=(q-1)|T_{1,2}(m,l)|,\quad\mbox{and}\\
|T'_{1,2}(m,l)|=(q-1)|T'_1(m,l)|.
\end{gather*}
\end{lemma}

\begin{proof}
If we fix a $c\in\lp\Fq\rp^{\times}$ and let $\mapf$ act on a coefficient triple $(A_1,A_2,A_3)$ defining 
a two-dimensional subspace of $F^3$ via \cref{1}, then the resulting two-dimensional subspace is given 
by the image of $\mapg$ (with the same $c$) acting on the original solution space.
Further, if 
\[
\deg{(A_1)}\ge\max\{\deg{(A_2)},\deg{(A_3)}\},
\]
then any polynomial solution $(P_1,P_2,P_3)$ to \cref{1} satisfies
\[
\deg{(P_1)}\le\max\{\deg{(P_2)},\deg{(P_3)}\}.
\]

We now see by \cref{SolTranslate} that, when restricted to $S_1(m)\times\lp\Fq\rp^\times$, 
$S_{1,2}(m)\times\lp\Fq\rp^\times$, $T_1(m,l)\times\lp\Fq\rp^\times$, and $T_{1,2}(m,l)\times\lp\Fq\rp ^\times$,
the function $\mapf$ has image 
$S_{1,3}(m)$, $S_{1,2,3}(m)$, $T_{1,3}(m,l)$, and $T_{1,2,3}(m,l)$, respectively. 
This together with \cref{Thunder1} completes the proof of the first four equations.

An entirely similar argument with 
\[
\phi '\big ( (Q_1,Q_2,Q_3),c\big ) = (Q_1,Q_2+cQ_1,Q_3)
\]
and
\[
\gamma '\big ( (Q_1,Q_2,Q_3),c\big ) = (Q_1-cQ_2,Q_2,Q_3)
\]
proves the last equation and completes the proof.
\end{proof}

\begin{prop}\label{Scounting}
For any integer $m\ge 0$ and all distinct $i,j\in\{1,2,3\}$
\[
|S(m)|=(q-1)N_{F^3}(m)=
\begin{cases}
q^{3m-2}(q^3-1)(q^2-1)&\text{if $m>0$,}\\
(q^3-1)&\text{if $m=0,$}
\end{cases}
\]
\[
|S_i(m)|=\frac{|S(m)|}{1+q+q^2},\qquad
|S_{i,j}(m)|=\frac{(q-1)|S(m)|}{1+q+q^2},
\]
and 
\[
|S_{1,2,3}(m)|=\frac{(q-1)^2|S(m)|}{1+q+q^2}.
\]

Fix an integer $l\ge 0$.
If $m> 2l$ then
\[
|S(m,l)|=q^{2m-l-1}(q-1)(q^2-1)N_{F^3}(l)=
\begin{cases}
q^{2(m+l)-3}(q^3-1)(q^2-1)^2&\text{if $l>0$,}\\
q^{2m-1}(q^3-1)(q^2-1)&\text{if $l=0$,}
\end{cases}
\]
\[
|S_i(m,l)|=\frac{|S(m,l)|}{1+q+q^2},\qquad
|S_{i,j}(m,l)|=\frac{(q-1)|S(m,l)|}{1+q+q^2},
\]
and 
\[
|S_{1,2,3}(m,l)|=\frac{(q-1)^2|S(m,l)|}{1+q+q^2}
\]
for all distinct $i,j\in \{1,2,3\}$.

Finally, if $m=2l>0$ then
\[
|S_R(m,l)|=\frac{|S_R(m)|(q-1)}{q}
\]
For all subsets $R\subseteq\{ 1,2,3\}$.
\end{prop}

\begin{proof}
As noted above, there is a $(q-1)$-to-1 correspondence between the elements of $S(m,l)$ and two-dimensional subspaces $V\subset F^3$ of height $m$ containing a one-dimensional subspace $W$ of height $l$.
Since $m>2l$, Minkowski's Theorem \cref{mink} implies that this subspace $W$ is unique to $V$. 
Therefore by \cref{lkheight} and \cref{2},
\begin{align*}
|S(m,l)|&=(q-1)
\sum_{\substack{W\subset F^3\\
\dim (W)=1\\
h(W)=l}}
{N_{F^3/W}(m-l)}\\
&=(q-1)q^{2(m-l)+l-1}(q^2-1)N_{F^3}(l).
\end{align*}
(Note that \cref{lkheight} applies since $m-l\ge 1$ if $l\ge 1$ and $m-l=m\ge 1$ if $l=0$.)
Next, by \cref{Thunder1} and \cref{precounting},
\begin{align*}
|S(m,l)| &=|S_1(m,l)|+|S_2(m,l)|+|S_3(m,l)|\\
&\qquad +|S_{1,2}(m,l)|+|S_{1,3}(m,l)|+
|S_{2,3}(m,l)|+|S_{1,2,3}(m,l)|\\
&=3|S_1(m,l)|+3|S_{1,2}(m,l)|+|S_{1,2,3}(m,l)|\\
&=3|S_1(m,l)|+(3+q-1)|S_{1,2}(m,l)|\\
&=3|S_1(m,l)|+(3+q-1)|S_{1,3}(m,l)|\\
&=\big ( 3+3(q-1)+(q-1)^2\big )|S_1(m,l)|\\
&=\big ( 1+q+q^2\big )|S_1(m,l)|.
\end{align*}
This together with \cref{2}, \cref{Thunder1}, and \cref{precounting} once more finishes the proof of the 
second part of the proposition. 

For the first part of the proposition, recall that every one-dimensional subspace of $F^3$ of height 
$m$ corresponds to $q-1$ elements of $S(m)$.
The rest of the proof follows exactly as above.

Next, suppose $m=2l>0$. Then by definition and \cref{mink} we have
\[
|S(m)|=\sum _{k=0}^l|S(m,k)|.
\]
By what we have already proven above
\[
|S(m,0)| = q^{2(m+0)-3}q^2(q^3-1)(q^2-1)= q^{2(m+0)-3}(q^3-1)(q^2-1)^2 + q^{2m -3}(q^3-1)(q^2-1).
\]
Using this together what we have already proven yields
\[
\begin{aligned}
|S(m)| & = \sum _{k=0}^{l-1}q^{2(m+k)-3}(q^3-1)(q^2-1)^2 
+ q^{2m-3}(q^3-1)(q^2-1) +|S(m,l)|\\
&=q^{2m-3}(q^3-1)(q^2-1)^2\frac{q^{2l}-1}{q^2-1}
+ q^{2m-3}(q^3-1)(q^2-1) +|S(m,l)|\\
&=q^{2m-3}(q^3-1)(q^2-1)q^{2l} +|S(m,l)|\\
&=q^{3m-3}(q^3-1)(q^2-1) +|S(m,l)|\\
&=|S(m)|q^{-1}+|S(m,l)|.
\end{aligned}
\]
This shows the case where $R=\emptyset$ of the last part of the proposition. The other cases are entirely similar.
\end{proof}

\begin{prop}\label{Tprimecounting}
Fix an integer $l\ge 0$. If $m> 2l$ then
\[
|T'(m,l)|=\frac{|S(m,l)|}{1+q+q^2}, \qquad
|T_1'(m,l)|=|T_2'(m,l)|=\frac{|S_1(m,l)|}{q+1},\qquad \mbox{and}
\]
\[
|T'_{1,2}(m,l)|=\frac{|S_{1,2}(m,l)|}{q+1}.
\]
\end{prop}

\begin{proof}
We have noted above that any one-dimensional subspace $W\subset F^3$ has a basis element consisting of a 
relatively prime triple of polynomials $(P_1,P_2,P_3)$ and the height of $W$ is the maximum of the 
degrees of these polynomials.
This triple is unique up to multiplication by an element of $\lp\Fq\rp^\times$, thus the indices 
$i$ where $\deg (P_i)=h(W)$ are uniquely determined by $W$.

We may now evaluate $|T'(m,l)|$ in a manner very similar to $|S(m,l)|$ above. Specifically, we have
\[
|T'(m,l)|=(q-1)\cdot
\sideset{}{'}\sum_{\substack{W\subset F^3\\
\dim (W)=1\\
h(W)=l}}
{N_{F^3/W}(m-l)},
\]
where the prime indicates that we are to sum only over those $W$ spanned by a relatively prime ordered triple 
counted in $S_3(l).$ Note that there are $q-1$ 
ordered triples in this set for each such $W$.
Via \cref{2}, \cref{lkheight}, and \cref{Scounting}

\begin{align*}
|T'(m,l)|&=(q-1)\cdot
\sideset{}{'}\sum_{\substack{W\subset F^3\\
\dim (W)=1\\
h(W)=l}}
{N_{F^3/W}(m-l)}\\
&=(q-1)q^{2(m-l)+l-1}(q^2-1)\frac{|S_3(l)|}{q-1}\\
&=q^{2m-l)-1}(q^2-1)\frac{S(l)}{1+q+q^2}\\
&=q^{2m-l-1}(q^2-1)q^2\frac{(q-1)N_{F^3}(l)}{1+q+q^2}\\
&=\frac{|S(m,l)|}{1+q+q^2}.
\numberthis\label{10}
\end{align*}

Next, by \cref{Thunder1} and \cref{precounting},
\begin{align*}
|T'(m,l)|&=|T'_1(m,l)|+|T'_2(m,l)|+|T'_{1,2}(m,l)|\\
&=(q+1)|T'_1(m,l)|.
\end{align*}
The remainder of the Proposition follows from this equation, \cref{10}, \cref{Thunder1}, and \cref{precounting}.
\end{proof}

\begin{prop}\label{Tcounting}
Fix an integer $l\ge 0$. If $m> 2l$ then
\begin{gather*}
|T(m,l)|=\frac{q^2|S(m,l)|}{1+q+q^2}, \qquad
|T_1(m,l)|=|T_3(m,l)|=\frac{q|S_1(m,l)|}{(q+1)},\\
|T_{1,2}(m,l)|=|T_{2,3}(m,l)|=\frac{q|S_{1,2}(m,l)|}{q+1}, \qquad\mbox{and}\qquad
|T_{1,2,3}(m,l)|=\frac{q|S_{1,2,3}(m,l)|}{q+1}.
\end{gather*}
\end{prop}

\begin{proof}
We argue almost exactly as above in the proof of \cref{Tprimecounting}, only this time
\[
|T(m,l)|=(q-1)\cdot
\sideset{}{'}\sum_{\substack{W\subset F^3\\
\dim (W)=1\\
h(W)=l}}
{N_{F^3/W}(m-l)},
\]
where the prime indicates that we are to sum only over those $W$ spanned by a relatively prime ordered triple 
counted in the union $S_2(l)\cup S_{1,2}(l)\cup S_{2,3}(l)\cup S_{1,2,3}(l)$. 
Via \cref{2}, \cref{lkheight} and \cref{Scounting}

\begin{align*}
|T(m,l)|&=(q-1)\cdot
\sideset{}{'}\sum_{\substack{W\subset F^3\\
\dim (W)=1\\
h(W)=l}}
{N_{F^3/W}(m-l)}\\
&=(q-1)q^{2(m-l)+l-1}(q^2-1)\frac{\big ( |S_2(l)|+|S_{1,2}(l)|+|S_{2,3}(l)|+|S_{1,2,3}(l)|\big )}{q-1}\\
&=q^{2m-l)-1}(q^2-1)\frac{|S(l)|}{1+q+q^2}( 1+2(q-1)+(q-1)^2)\\
&=q^{2m-l-1}(q^2-1)q^2\frac{(q-1)N_{F^3}(l)}{1+q+q^2}\\
&=\frac{q^2}{1+q+q^2}|S(m,l)|.
\numberthis\label{13}
\end{align*}

Next, by \cref{Thunder1} and \cref{Tprimecounting}
\begin{align*}
\frac{(q-1)|S(m,l)|}{(q+1)(1+q+q^2)} &= |T'_{1,2}(m,l)|\\
&= |S_{1,2}(m,l)|-|T_{1,2}(m,l)|\\
&= \frac{(q-1)|S(m,l)|}{1+q+q^2} - |T_{1,2}(m,l)|,
\end{align*}
so that
\[ 
|T_{1,2}(m,l)|=\frac{q(q-1)|S(m,l)|}{(q+1)(1+q+q^2}.
\]

Using this together with \cref{Thunder1} and \cref{precounting} yields
\begin{align*}
|T(m,l)|&=|T_1(m,l)|+|T_3(m,l)|+|T_{1,3}(m,l)|+|T_{1,2}(m,l)|+|T_{2,3}(m,l)|+|T_{1,2,3}(m,l)|\\
&=(q+1)|T_1(m,l)|+(q+1)|T_{1,2}(m,l)|\\
&=(q+1)|T_1(m,l)|+\frac{q(q-1)|S(m,l)|}{1+q+q^2}.
\end{align*}
The remainder of the Proposition follows from this equation, \cref{13}, \cref{Thunder1}, \cref{precounting}, and
\cref{Scounting}.
\end{proof}

\subsection{Main counting results}\label{main}

\begin{thm}\label{firstone}
Fix integers $m, l\ge 0$. Let $N_1(m,l)$ denote the number of ordered triples 
$(A_1,A_2,A_3)$ of polynomials such that:
\begin{itemize}
\item
$A_1,A_2,A_3$ are relatively prime;
\item
$A_1$ is monic of degree $m$;
\item
the degrees of both $A_2$ and $A_3$ are no
greater than $m$; and
\item
there is a solution to \cref{1} of height $l$.
\end{itemize}
If $m>2l$ then
\[
N_1(m,l)
=\begin{cases}
(q^2-1)^2q^{2(m+l)-1}&\text{if $l>0$,}\\
(q^2-1)q^{2m+1}&\text{if $l=0$.}
\end{cases}
\]
If $m=2l>0$ then
\[
N_1(m,l)=(q-1)(q^2-1)q^{3m-1}.
\]

Further, suppose $m>2l\ge 0$ and let $N_1'(m,l)$ denote the number of those ordered triples 
as above but where 
the solution $(P_1,P_2,P_3)$
of height $l$ satisfies $\deg (P_2)<l$. Then
\[
N_1'(m,l)= \frac{1}{q+1}N_1(m,l)=
\begin{cases}
(q^2-1)(q-1)q^{2(m+l)-1}&\text{if $l>0,$}\\
(q-1)q^{2m+1}&\text{if $l=0.$}
\end{cases}
\]
\end{thm}

\begin{proof}
Since we demand that $A_1$ is monic, in the notation of the previous subsection we have
\[
N_1(m,l)
=\frac{1}{q-1}\big ( |S_1(m ,l)|
+|S_{1,2}(m ,l)|
+|S_{1,3}(m ,l)|
+|S_{1,2,3}(m ,l)|\big ) .
\]
We now apply \cref{Scounting}. If $m>2l$ we have
\[
\begin{aligned}
N_1(m,l)&=
\frac{1}{q-1}\big ( |S_1(m,l)|+|S_{1,2}(m,l)|+|S_{1,3}(m,l)+|S_{1,2,3}(m,l)|\big ) \\
&=\frac{1+2(q-1)+(q-1)^2}{(q-1)(1+q+q^2)}|S(m,l)|\\
&=\frac{q^2}{q^3-1)}|S(m,l)|\\
&=
\begin{cases}
(q^2-1)^2q^{2(m+l)-1}&\text{if $l>0,$}\\
(q^2-1)q^{2m+1}&\text{if $l=0$.}
\end{cases}
\end{aligned}
\]
The case $m=2l$ is entirely similar.

The second part we have
\begin{align*}
N_1'(m,l)
&=\frac{1}{q-1}\big ( |S_1(m ,l)|-|T_1(m,l)|
+|S_{1,2}(m,l)|-|T_{1,2}(m,l)|\\
&\qquad +|S_{1,3}(m,l)|-|T_{1,3}(m,l)|
+|S_{1,2,3}(m,l)|-|T_{1,2,3}(m,l)|\big )\\
&=N_1(m,l)\left ( 1 -\frac{q}{q+1}\right )\\
&=\frac{1}{q+1}N_1(m,l)
\end{align*}
by \cref{Tcounting} and our expression for $N_1(m,l)$ above. 
\end{proof}

\begin{thm}\label{secondone}
Fix integers $m,l\ge 0$. Let $N_2(m,l)$ denote the number
of ordered triples $(A_1,A_2,A_3)$ of polynomials such that:
\begin{itemize}
\item
$A_1,A_2,A_3$ are relatively prime;
\item
$A_1$ is monic of degree $m$;
\item
the degrees of both $A_2$ and $A_3$ are strictly less than $m$; and
\item
there is a solution to \cref{1} of height $l$.
\end{itemize}
If $m> 2l$ then
\[
N_2(m,l)
=\begin{cases}
(q^2-1)^2q^{2(m+l)-3}&\text{if $l>0$,}\\
(q^2-1)q^{2m-1}&\text{if $l=0$.}
\end{cases}
\]
If $m=2l>0$ then
\[
N_2(m,l)=(q-1)(q^2-1)q^{3m-3}.
\]

Further, suppose $m>2l\ge 0$ and let $N_2'(m,l)$ denote the number of those triples as above but where the 
solution $(P_1,P_2,P_3)$
satisfies $\deg (P_2)<l.$ Then
\[
N_2'(m,l)= \frac{1}{q+1}N_2(m,l)=
\begin{cases}
(q-1)(q^2-1)q^{2(m+l)-3}&\text{if $l>0$,}\\
(q-1)q^{2m-1}&\text{if $l=0$.}
\end{cases}
\]
\end{thm}

\begin{proof}
Since we demand that $A_1$ is monic, in the notation of the previous section we have
\[
\begin{aligned}
N_2(m,l)&=\frac{1}{q-1}|S_1(m,l)|\\
&=\frac{1}{(q-1)(1+q+q^2)}|S(m,l)|\\
&=\begin{cases}
\frac{(q^3-1)(q^2-1)(q+1)}{1+q+q^2}q^{2(m+l)-3}&\text{if $l>0$,}\\
\frac{(q+1)(q^3-1)}{1+q+q^2}q^{2m-1}&\text{if $l=0$,}
\end{cases} \\
&=\begin{cases}
(q^2-1)^2q^{2(m+l)-3}&\text{if $l>0$,}\\
(q^2-1)q^{2m-1}&\text{if $l=0$,}
\end{cases}
\end{aligned}
\]
by \cref{Scounting}

For the second part, note that since the degree of $A_1$ is strictly larger than the degrees of both
$A_2$ and $A_3$, any solution $(P_1,P_2,P_3)$ of height $l$ to \cref{1} with $\deg (P_2)<l$ necessarily
must have $\deg (P_3)=l$ strictly larger than the degrees of both $P_1$ and $P_2$. Thus
\[
\begin{aligned}
N_2'(m,l)&=\frac{1}{q-1}|T'_1(m,l)| \\
&=\frac{1}{q^2-1}|S_1(m,l)|
\end{aligned}
\]
by \cref{Tprimecounting}. 
\end{proof}

\begin{thm}\label{thirdone}
Fix integers $m,l\ge 0$. Let
$N_3(m,l)$ denote the number
of ordered triples $(A_1,A_2,A_3)$ of polynomials such that:
\begin{itemize}
\item
$A_1,A_2,A_3$ are relatively prime;
\item
$A_1$ is monic of degree $m$;
\item
$\deg (A_2)\le m$;
\item
$\deg (A_3)<m$; and
\item
there is a solution to \cref{1} of height $l$.
\end{itemize}
If $m>2l$ then
\[
N_3(m,l)=
\begin{cases}
(q^2-1)^2q^{2(m+l)-2}&\text{if $l>0,$}\\
(q^2-1)q^{2m}&\text{if $l=0.$}
\end{cases}
\]
If $m=2l>0$ then
\[
N_3(m,l)= (q-1)(q^2-1)q^{3m-2}.
\]

Further, suppose $m>2l\ge 0$ and let $N_3'(m,l)$ denote the number of those triples above but 
where the solution $(P_1,P_2,P_3)$
satisfies $\deg (P_2)<l$. Then
\[
N_3'(m,l)=\frac{1}{q+1}N_3(m,l)=
\begin{cases}
(q^2-1)(q-1)q^{2(m+l)-1}&\text{if $l>0,$}\\
(q-1)q^{2m}&\text{if $l=0.$}
\end{cases}
\]
\end{thm}

\begin{proof}
Since we demand that $A_1$ is monic, in the notation of the previous section we have
\[
\begin{aligned}
N_3(m,l)&=\frac{1}{q-1}\big ( |S_1(m,l)|+|S_{1,2}(m,l)|\big ) \\
&=\frac{q}{(q-1)(1+q+q^2)}|S(m,l)| \\
&=
\begin{cases}
(q^2-1)^2q^{2(m+l)-1}&\text{if $l>0$,}\\
(q^2-1)q^{2m}&\text{if $l=0$}
\end{cases}
\end{aligned}
\]
by \cref{Scounting}.
For the second part, note that since the degree of $A_1$ is strictly larger than the degree of $A_3$
and at least as large as the degree of $A_2$, any solution $(P_1,P_2,P_3)$ of height $l$ to
\cref{1} with $\deg (P_2)<l$ necessarily must have $\deg (P_3)=l$ strictly larger than the degrees
of both $P_1$ and $P_2$. Thus 
\[
\begin{aligned}
N_3'(m,l)&=\frac{1}{q-1}\big ( |T'_1(m,l)|+|T'_{1,2}(m,l)|\big ) \\
&=\frac{1}{q^2-1}\big ( |S_1(m,l)|+|S_{1,2}(m,l)|\big ) 
\end{aligned}
\]
by \cref{Tprimecounting}.
\end{proof}

\begin{prop}\label{pcounting}
For all $g>0$ and $\epsilon\in \{ 1,2\}$ we have
\[
|\cPe (g)|=q^{2g+\epsilon }(1-q^{-2}).
\]
Also,
\[
|\cPe (0)|=q^{\epsilon}(1-q^{-1})
\]
for $\epsilon \in \{ 1,2\}$.

Finally,
\[
|\cPe '(g)|=
\begin{cases}
q^{2g+\epsilon }(1-q^{-1})&\text{if $g>0$ or $g=0$ and $\epsilon = 2$,}\\
q&\text{if $g=0$ and $\epsilon = 1$.}
\end{cases}
\]
\end{prop}

\begin{proof} We will prove the cases for $\cPe (g)$; those for $\cPe '(g)$ are well-known (see Proposition
2.3 of \cite{Rosen}, for example).
The case where $g=0$ is trivial, so we will assume that $g>0$.

Suppose that $g\equiv 0\pmod 3$. If $\epsilon = 1$, then by \cref{gcd_cf} and definition 
$f\in \Fq [X]$ is cube-free of degree $2g+1$ if and only if
$\big ( c_0(f),c_1(f),c_2(f)\big ) \in S_2(2g/3)\cup S_{1,2}(2g/3)$. 
The result in this case follows from \cref{Scounting} and \cref{2}.
If $\epsilon = 2$, then by \cref{gcd_cf} and definition
$f\in \Fq [X]$ is cube-free of degree $2g+2$ if and only
$\big ( c_0(f),c_1(f),c_2(f)\big )\in S_3(2g/3)\cup S_{1,3}(2g/3)\cup S_{2,3}(2g/3)\cup S_{1,2,3}(2g/3)$, and the result 
again follows from \cref{Scounting} and \cref{2}. The other cases
($g\equiv 1\pmod 3$ and $g\equiv 2\pmod 3$) are proven in an entirely similar manner.
\end{proof}

\section{The \textit{a}-numbers of hyperelliptic curves in characteristic three}
\label{consequences}

In this section we restrict to the case $p=3$ to deduce the results mentioned in \cref{intro}.

\subsection{The proportion of curves with a given \textit{a}-number}
We may now prove \cref{bigone} via the counting results in the previous section and 
\cref{anumberheights}.

\begin{proof}[Proof of \cref{bigone}]

Suppose first that $g\equiv 0\pmod 3$. 

For $f\in\cPe (g)$, if $\epsilon = 1$ then the space $V$ in \cref{anumberheights} has height
\[
h(V)=\frac{2g}{3}=\deg \big ( c_1(f)\big )>\deg\big ( c_2(f)\big ) .
\]
 We use \cref{thirdone} with 
\[
A_1=c_1(f),\qquad A_3=c_2(f),\qquad A_2=c_0(f),\qquad P_1=c_2(Q), \qquad P_2=c_1(Q),\qquad \text{and}\  
P_3=c_0(Q).
\]
We set $m=\frac{2g}{3}$. Then for any $0\le l\le \frac{g}{3}$ the number of $f\in\cPe (g)$ 
with $a(H_f)=\frac{g}{3}-l=a$ is $N_3(m,l)$ by \cref{anumberheights}. We thus have
\[
\begin{aligned}
\mu _{1 ,g}(a)&= \frac{N_3(m, g/3 - a)}{|\cPe (g)|}\\
&= \begin{cases}
\frac{q^{3m-2-2a}(q^2-1)(q-1)}{q^{2g-1}(q^2-1)}&\text{if $a=0,$}\\
\frac{q^{2(m+g/3-a)-2}(q^2-1)^2}{q^{2g-1}(q^2-1)}&\text{if $0<a<g/3$,}\\
\frac{q^{2(m+g/3-a)}(q^2-1)}{q^{2g-1}(q^2-1)}&\text{if $a=g/3$,}
\end{cases}\\
&=\begin{cases}
q^{-2a}\frac{q-1}{q}&\text{if $a=0$,}\\
q^{-2a}\frac{q^2-1}{q}&\text{if $0<a<g/3$,}\\
q^{-2a}q&\text{if $a=g/3$.}
\end{cases}
\end{aligned}
\]
 
For $f\in\cPe (g)$, if $\epsilon = 2$ then the space $V$ in \cref{anumberheights} has height
\[
h(V)=\frac{2g}{3}=\deg\big ( c_2(f)\big ) .
\]
We use \cref{firstone} with 
\[
A_1=c_2(f),\qquad A_3=c_1(f),\qquad A_2=c_0(f),\qquad P_1=c_0(Q),\qquad P_2=c_2(Q),\qquad\text{and}\  
P_3=c_1(Q).
\]
We set $m=\frac{2g}{3}$. Then for any $0\le l\le \frac{g}{3}$ the number of $f\in\cPe (g)$ 
with $a(H_f)=\frac{g}{3}-l=a$ is $N_1(m,l)$ by \cref{anumberheights}. We thus have
\[
\begin{aligned}
\mu _{2,g}(a)&=\frac{N_1(m, g/3 -a)}{|\cPe (g)|}\\
&=\begin{cases}
\frac{q^{3m-1}(q^2-1)(q-1)}{q^{2g}(q^2-1)}&\text{if $a=0,$}\\
\frac{q^{2(m+g/3-a) -1}(q^2-1)^2}{q^{2g}(q^2-1)}&\text{if $0<a<g/3,$}\\
\frac{q^{2(m+g/3-a) +1}(q^2-1)}{q^{2g}(q^2-1)}&\text{if $a=g/3,$}
\end{cases}\\
&=\begin{cases}
q^{-2a}\frac{q-1}{q}&\text{if $a=0$,}\\
q^{-2a}\frac{q^2-1}{q}&\text{if $0<a<g/3,$}\\
q^{-2a}q&\text{if $a=g/3$.}\end{cases}
\end{aligned}
\]

Next suppose that $g\equiv 1\pmod 3$. 

For $f\in\cPe (g)$, if $\epsilon = 1$ then the space $V$ in \cref{anumberheights} has height
\[
h(V)=\frac{2g+1}{3}=\deg \big ( c_0(f)\big ) >\max\{ \deg\big ( c_1(f)\big ) ,\deg\big ( c_2(f)\big )\}.
\]
We use \cref{secondone} with
\[
A_1=c_0(f),\qquad A_2=c_1(f),\qquad A_3=c_2(f),\qquad P_1=c_2(Q),\qquad P_2=c_1(Q),\qquad\text{and}\  
P_3=c_0(Q).
\]
We set $m=\frac{2g+1}{3}$. 
Note that $m=2(g-1)/3 +1,$ so that by \cref{mink} the first minima $\mu _1(V)$
is no greater than $(g-1)/3$. By
\cref{anumberheights}
for any $0<l\le \frac{g-1}{3}$ the number of $f\in\cPe(g)$ with 
$a(H_f)=\frac{g-1}{3}-l+1=a$ is
$N'_2(m,l)+\big (N_2(m,l-1) -N'_2(m,l-1)\big )$, the number of $f\in\cPe(g)$ with $a(H_f)=\frac{g-1}{3}+1$ is
$N'_2\big ( m, 0)$, and the number of $f\in \cPe(g)$ with $a(H_f)=0$ is 
$N_2(m,\frac{g-1}{3})-N_2'(m,\frac{g-1}{3})$. 

Excluding for the moment the case where $g=1$ and $a=0$, we therefore have 
\[
\begin{aligned}
\mu _{1,g}(a)&=\begin{cases}
               \frac{N_2\big ( m,(g-1)/3\big ) -N_2'\big ( m,(g-1)/3\big )}{|\cPe (g)|}
               &\text{if $a=0,$}\\
               \frac{N_2'\big ( m,(g-1)/3-a+1\big )+N_2\big ( m,(g-1)/3-a\big ) -N'_2\big ( m,(g-1)/3-a\big )}{|\cPe (g)|}
               &\text{if $1\le a\le\frac{g-1}{3},$}\\
               \frac{N'_2( m,0)}{|\cPe (g)|}&\text{if $a=\frac{g+2}{3},$}
              \end{cases}\\
     &=\begin{cases}
       \frac{q^{2g-2}(q^2-1)^2}{(q+1)q^{2g-1}(q^2-1)}&\text{if $a=0$,}\\
       \frac{q^{2(m+(g-1)/3 -a) -3}q(q^2-1)^2}{q^{2g-1}(q^2-1)}&\text{if $0<a<\frac{g-1}{3}$,}\\
       \frac{q^{2(m+(g-1)/3-a)-1}(q^2+q-1)}{(q+1)q^{2g-1}}&\text{if $a=\frac{g-1}{3}$,}\\
       \frac{(q-1)q^{2m-1}}{(q^2-1)q^{2g-1}}&\text{if $a= \frac{g+2}{3}$,}
       \end{cases}\\
     &=\begin{cases}
       q^{-2a}\frac{q-1}{q}&\text{if $a=0$,}\\
       q^{-2a}\frac{q^2-1}{q}&\text{if $0< a<\frac{g-1}{3}$,}\\
       q^{-2a}\frac{q^2+q-1}{q+1}&\text{if $a=\frac{g-1}{3},$}\\
       q^{-2a}\frac{q^2}{q+1}&\text{if $a=\frac{g+2}{3}.$}
       \end{cases}
\end{aligned}
\]

If $g=1$ (so that $(g-1)/3=0$ and $m=1$) and $a=0$ we have
\[
\begin{aligned}
\mu _{1,1}(0)&=\frac{N_2\big ( 1,0\big ) -N_2'\big ( 1, 0\big )}{|\cPe (1)|}\\
              &= \frac{(q^2-1)q-(q-1)q}{q(q^2-1)}\\
              &= \frac{q}{q+1}.
\end{aligned}
\]

For $f\in\cPe (g)$, if $\epsilon = 2$ then the space $V$ in \cref{anumberheights} has height
\[
h(V)=\frac{2g+1}{3}=\deg (c_1)>\deg (c_2).
\]
We use \cref{thirdone} with
\[
A_1=c_1(f),\qquad A_2=c_0(f),\qquad A_3=c_2(f),\qquad P_1=c_1(Q),\qquad P_2=c_2(Q),\qquad\text{and}\  
P_3=c_0(Q).
\]
We set $m=\frac{2g+1}{3}$. 
Once again the first minima $\mu _1(V)$
is no greater than $(g-1)/3$. By
\cref{anumberheights}
for any $0<l\le \frac{g-1}{3}$ the number of $f\in\cPe(g)$ with 
$a(H_f)=\frac{g-1}{3}-l+1=a$ is
$N'_3(m,l)+\big ( N_3(m,l-1)-N'_3(m,l-1)\big )$. The number of $f\in\cPe(g)$ with $a(H_f)=
\frac{g-1}{3}+1$ is
$N_3'(m,0)$ and the number of $f\in\cPe (g)$ with $a(H_f)=0$ is
$N_3(m,\frac{g-1}{3})-N_3'(m,\frac{g-1}{3})$. 

Again temporarily excluding the case where $g=1$ and $a=0$, we therefore have
\[
\begin{aligned}
\mu _{2,g}(a)&=\begin{cases}
               \frac{N_3\big ( m,(g-1)/3\big ) -N_3'\big ( m,(g-1)/3\big )}{|\cPe (g)|}
               &\text{if $a=0,$}\\
               \frac{N_3'\big ( m,(g-1)/3-a+1\big )+N_3\big ( m,(g-1)/3-a\big ) -
N'_3\big ( m, (g-1)/3-1\big )}{|\cPe (g)|}
               &\text{if $0<a\le\frac{g-1}{3},$}\\
               \frac{N'_3 ( m,0)}{|\cPe (g)|}&\text{if $a=\frac{g+2}{3},$}
              \end{cases}\\
     &=\begin{cases}
       \frac{q^{2g-1}(q^2-1)^2}{(q+1)q^{2g}(q^2-1)}&\text{if $a=0$,}\\
       \frac{q^{2(m+(g-1)/3 -a +1) -3}(q^2-1)^2}{q^{2g}(q^2-1)}&\text{if $0<a<\frac{g-1}{3}$,}\\
       \frac{q^{2(m+(g-1)/3-a)}(q^2+q-1)}{(q+1)q^{2g}}&\text{if $a=\frac{g-1}{3}$,}\\
       \frac{q^{2(m+(g-1)/3 -a +1)}(q-1)}{q^{2g}(q^2-1)}&\text{if $a= \frac{g+2}{3}+1$,}
       \end{cases}\\
     &=\begin{cases}
       q^{-2a}\frac{q-1}{q}&\text{if $a=0$,}\\
       q^{-2a}\frac{q^2-1}{q}&\text{if $0<a<\frac{g-1}{3}$,}\\
       q^{-2a}\frac{q^2+q-1}{q+1}&\text{if $a=\frac{g-1}{3}$,}\\
       q^{-2a}\frac{q^2}{q+1}&\text{if $a=\frac{g+2}{3}.$}
       \end{cases}
\end{aligned}
\]

If $g=1$ and $a=0$ we have
\[
\begin{aligned}
\mu _{2,1}(0)&=\frac{N_3\big ( 1,0\big ) -N_3'\big ( 1, 0\big )}{|\cPe (1)|}\\
              &= \frac{(q^2-1)q^2-(q-1)q^2}{q^2(q^2-1)}\\
              &= \frac{q}{q+1}.
\end{aligned}
\]

Finally suppose that $g\equiv 2\pmod 3$.

For $f\in\cPe (g)$, if $\epsilon = 1$ then the space $V$ in \cref{anumberheights} has height
\[
h(V)=\frac{2g-1}{3}=\deg (c_2).
\]
We use \cref{firstone} with
\[
A_1=c_2(f),\qquad A_2=c_0(f),\qquad A_3=c_1(f),\qquad P_1=c_0(Q),\qquad P_2=c_2(Q),\qquad\text{and}\  
P_3=c_1(Q).
\]
We set $m=\frac{2g-1}{3}$. 
Note that $m=2(g-2)/3 +1,$ so that by \cref{mink} the first minima $\mu _1(V)$
is no greater than $(g-2)/3$. By
\cref{anumberheights}
for any $0<l\le \frac{g-2}{3}$ the number of $f\in\cPe(g)$ with 
$a(H_f)=\frac{g-2}{3}-l+1=a$ is
$N'_1(m,l)+\big ( N_1(m,l-1)-N'_1(m,l-1)\big )$. The number of $f\in\cPe(g)$ with $a(H_f)=\frac{g-2}{3}+1$ is
$N_1'(m,0),$ and the number of $f\in\cPe (g)$ with $a(H_f)=0$ is
$N_1(m,\frac{g-2}{3})-N_1'(m,\frac{g-2}{3})$.

Temporarily excluding the case where $g=2$ and $a=0$, we therefore have
\[
\begin{aligned}
\mu _{1,g}(a)&=\begin{cases}
               \frac{N_1\big ( m,(g-2)/3\big ) -N_1'\big ( m,(g-2)/3\big )}{|\cPe (g)|}
               &\text{if $a=0,$}\\
               \frac{N_1'\big ( m,(g-2)/3-a+1\big )+N_1\big ( m, (g-2)/3-a\big )
-N_1'\big ( m, (g-2)/3-a\big )}{|\cPe (g)|}
               &\text{if $0<a\le\frac{g-2}{3},$}\\
               \frac{N_1' ( m,0 )}{|\cPe (g)|}&\text{if $a=\frac{g+1}{3},$}
              \end{cases}\\
     &=\begin{cases}
       \frac{q^{2g-4+2}(q^2-1)^2}{(q+1)q^{2g-1}(q^2-1)}&\text{if $a=0$,}\\
       \frac{q^{2(m+(g-2)/3 -a}(q^2-1)^2}{q^{2g-1}(q^2-1)}&\text{if $0<a< \frac{g-2}{3}$,}\\
       \frac{q^{2(m+(g-2)/3 -a)+1}(q^2+q-1)}{(q+1)q^{2g-1}}&\text{if $a= \frac{g-2}{3}$,}\\
       \frac{q^{2(m+(g-2)/3 -a +1) +1}(q^2-1)}{q^{2g-1}(q^2-1)}&\text{if $a= \frac{g+1}{3}$,}
       \end{cases}\\
     &=\begin{cases}
       q^{-2a}\frac{q-1}{q}&\text{if $a=0$,}\\
       q^{-2a}\frac{q^2-1}{q}&\text{if $0< a<\frac{g-2}{3}$,}\\
       q^{-2a}\frac{q^2+q-1}{q+1}&\text{if $a=\frac{g-2}{3}$,}\\
       q^{-2a}\frac{q^2}{q+1}&\text{if $a=\frac{g+1}{3}.$}
       \end{cases}
\end{aligned}
\]

If $g=2$ (so that $(g-2)/3=0$ and $m=1$) and $a=0$ we have
\[
\begin{aligned}
\mu _{1,2}(0)&=\frac{N_1\big ( 1,0\big ) -N_1'\big ( 1, 0\big )}{|\cPe (2)|}\\
              &= \frac{(q^2-1)q^3-(q-1)q^3}{q^3(q^2-1)}\\
              &= \frac{q}{q+1}.
\end{aligned}
\]

For $f\in\cPe (g)$, if $\epsilon = 2$ then the space $V$ in \cref{anumberheights} has height
\[
h(V)=\frac{2g+2}{3}=\deg (c_0)>\max\{ \deg (c_1),\deg (c_2)\}.
\]
We use \cref{secondone} with
\[
A_1=c_0,\qquad A_2=c_1,\qquad A_3=c_2,\qquad P_1=Q_1,\qquad P_2=Q_2,\qquad\text{and}\  P_3=Q_3.
\]
We set $m=\frac{2g+2}{3}$. 
Note that $m=2(g+1)/3,$ so that by \cref{mink} the first minima $\mu _1(V)$
is no greater than $(g+1)/3$. By
\cref{anumberheights} $a(H_f)>0$ only if $\mu _1(V) \le \frac{g-2}{3}$. Moreover,
for any $0\le l\le\frac{g-2}{3}$ the number of $f\in\cPe(g)$ with 
$a(H_f)=\frac{g-2}{3}-l+1=a$ is
$N_2(m,l)$. The number of $f\in\cPe(g)$ with $a(H_f)=0$ is $N_2\big ( m, (g+1)/3\big )$. 

Therefore
\[
\begin{aligned}
\mu _{2,g}(a)&=
               \frac{N_2\big ( m,(g-2)/3 -a+1\big )}{|\cPe (g)|}\\
     &=\begin{cases}
       \frac{q^{2g-1}(q^2-1)(q-1)}{q^{2g}(q^2-1)}&\text{if $a=0$,}\\
       \frac{q^{2(m+(g-2)/3-a+1)-3}(q^2-1)^2}{q^{2g}(q^2-1)}&\text{if $1\le a<\frac{g+1}{3}$,}\\
       \frac{q^{2(m+(g-2)/3 -a +1) -1}(q^2-1)}{q^{2g}(q^2-1)}&\text{if $a= \frac{g+1}{3}$,}
       \end{cases}\\
     &=\begin{cases}
       q^{-2a}\frac{q-1}{q}&\text{if $a=0$,}\\
       q^{-2a}\frac{q^2-1}{q}&\text{if $1\le a<\frac{g+1}{3}$,}\\
       q^{-2a}q&\text{if $a=\frac{g+1}{3}$.}
       \end{cases}
\end{aligned}
\]

\end{proof}

We now turn to proofs of \cref{actualvalues} and \cref{boundedvalues}.

\begin{proof}[Proof of \cref{actualvalues}]
Part~\hyperref[squarefreevanishing]{(\ref*{squarefreevanishing})} is immediate from \cref{bigone}.
To prove part~\hyperref[topanumber]{(\ref*{topanumber})}, let $\epsilon\in\lb1,2\rb$ and $f\in\mathcal{P}_\epsilon(q,g)$, so there are monic squarefree polynomials $f_1, f_2\in\Fq[X]$ such that $f=f_1\lp f_2\rp^2$ and $\gcd{\lp f_1,f_2\rp}=1$.

Suppose for a moment that we are in the case where $f$ is not squarefree, so that $\deg{f_2}\geq1$.
Let $g_1$ be the integer such that $2g_1+\epsilon=\deg{f_1}$; since $g\equiv1\pmod{3}$ and $g-1\geq g_1$, we see that
\[
\left \lceil \frac{g}{3}\right \rceil >
\frac{g-1}{3}
=\left \lceil \frac{g-1}{3}\right  \rceil
\geq\left \lceil\frac{g_1}{3}\right \rceil .
\]
Since $a(H_f)=a(H_{f_1})$ by \cref{collins}, we apply \cref{bigone} to deduce that 
$a(H_f)<\left \lceil \frac{g}{3}\right \rceil$ in this case.
Thus, we see that
\[
\mu_{\epsilon,g}'(a)
=\mu_{\epsilon,g}(a)
\lp\frac{q^{2g+\epsilon}-q^{2g+\epsilon-2}}
{q^{2g+\epsilon}-q^{2g+\epsilon-1}}\rp
=\mu_{\epsilon,g}(a)\lp1+q^{-1}\rp,
\]
so part~\hyperref[topanumber]{(\ref*{topanumber})} is true by \cref{bigone}.
\end{proof}

\begin{proof}[Proof of \cref{boundedvalues}]
Let
\[
I=\lb f\in\mathcal{P}_{q,\epsilon}(g)\setminus\mathcal{P}'_{q,\epsilon}(g)\mid a(H_f)=a\rb.
\]

To prove part~\hyperref[corpart1]{(\ref*{corpart1})}, we remark that the $g=0$ case is immediate by \cref{bigone}, so suppose that $g>0$.
Then
\begin{align*}
\mu_{\epsilon,g}'(a)
&=\mu_{\epsilon,g}(a)
\lp\frac{q^{2g+\epsilon}-q^{2g+\epsilon-2}}
{q^{2g+\epsilon}-q^{2g+\epsilon-1}}\rp
-\frac{1}{q^{2g+\epsilon}-q^{2g+\epsilon-1}}\sum_{f\in I}{1}\\
&=\mu_{\epsilon,g}(a)
\lp1+q^{-1}\rp
-\frac{1}{q^{2g+\epsilon}\lp1-q^{-1}\rp}\sum_{f\in I}{1},
\end{align*}
so part~\hyperref[corpart1]{(\ref*{corpart1})} follows from the fact that
\[
\lv I\rv
\leq \lv\mathcal{P}_{q,\epsilon}(g)\setminus\mathcal{P}'_{q,\epsilon}(g)\rv
=q^{2g+\epsilon-1}\lp1-q^{-1}\rp.
\]

Turning to part~\hyperref[corpart2]{(\ref*{corpart2})}, we see once again that if $g=0$, then the result is immediate from \cref{bigone}; thus, we suppose $g>0$.
Since $a>0$ we use \cref{bigone} to see that for all $d\in\lb0,\ldots,g\rb$,
\[
\mu_{\epsilon,d}(a)\leq q^{-2a+1},
\]
so part~\hyperref[corpart1]{(\ref*{corpart1})} implies that for all such $d$,
\[\numberthis\label{upperbou}
\mu_{\epsilon,d}'(a)-q^{-2a+1}\leq q^{-2a}.
\]
Thus, part~\hyperref[corpart2]{(\ref*{corpart2})} will follow if we show
\[
\lv I\rv<2q^{2g+\epsilon-2a}\lp1-q^{-1}\rp.
\]
To do so, note that
\begin{align*}
\lv I\rv
&=\sum_{d=0}^{g-1}
{\sum_{\substack{f_1\in\mathcal{P}'_{q,\epsilon}(g)\\a(H_f)=a}}
{\sum_{\substack{f_2\in\Fq[X]\text{ with }f\text{ monic,}\\\text{squarefree, of degree }g-d,\\\text{ and coprime to }f_1}}
{1}}}&&\text{by \cref{collins}}\\
&\leq\sum_{d=0}^{g-1}
{\sum_{\substack{f_1\in\mathcal{P}'_{q,\epsilon}(g)\\a(H_f)=a}}
{\sum_{\substack{f_2\in\Fq[X]\text{ with }f\text{ monic}\\\text{of degree }g-d}}
{1}}}\\
&=\sum_{d=0}^{g-1}{\mu_{\epsilon,d}'(a)\lp q^{2d+\epsilon}-q^{2d+\epsilon-1}\rp q^{g-d}}
&&\text{since }\mu_{\epsilon,0}'(a)=0\\
&\leq q^{g+\epsilon-2a+1}\lp1+q^{-1}\rp
\sum_{d=0}^{g-1}{\lp q^d-q^{d-1}\rp}
&&\text{by \cref{upperbou}}\\
&=q^{g+\epsilon-2a}\lp\frac{1+q^{-1}}{1-q^{-1}}\rp\lp q^g-q^{g-1}-1+q^{-1}\rp\\
&<q^{2g+\epsilon-2a}\lp\frac{1+q^{-1}}{1-q^{-1}}\rp\lp1-q^{-1}\rp\\
&\leq2q^{2g+\epsilon-2a}\lp1-q^{-1}\rp ,
\end{align*}
since $q\geq3$ implies $\lp\frac{1+q^{-1}}{1-q^{-1}}\rp\leq2.$
\end{proof}

\subsection{A heuristic model and conjecture}\label{heuristicmodel}

We continue to consider the $p=3$ case.
As in \cref{heuristic}, we write $F$ for $\Fq(X)$.
In \cref{anumberheights} we associated to any monic squarefree polynomial $f\in\Fq[X]$ a particular 2-dimensional 
subspace $V$ of $F^3$.
Suppose that $g$ is a positive multiple of 3, say $g=3j$ for $j\in\Z_{>0}$. In this case 
for any $f\in\mathcal{P}'_2(g)$:
\begin{itemize}
\item
$h\lp V\rp=2j$, and
\item
$a\lp H_f\rp=j-\mu _1(V)$.
\end{itemize}
Thus, we propose to model the $a$-number statistics of hyperelliptic curves of a given genus by random 2-dimensional subspaces of a fixed height.
To this end, for any $j,a\in\Z_{\geq0}$, let
\[
\nu_{j}(a)
=\frac{\lv\lb V\subset F^3\mid\dim_F (V)=2,\ h(V)=2j,\ \mu _1(V)=j-a\rb\rv}
{\lv\lb V\subset F^3\mid \dim _F (V)=2,\ h(V)=2j\rb\rv}.
\]

In the notation of \cref{thunder}
\[
\nu _j(a)=\frac{|S(2j,j-a)|}{|S(2j)|}
\]
whenever $0\le a\le j$, and 0 otherwise, since $0\le\mu _1(V)\le\frac{h(V)}{2}$ for all 2-dimensional subspaces by \cref{mink}. 
By \cref{Scounting} we immediately have the following.

\begin{thm}\label{heurcomp}
For any prime $p$, $j\in\Z_{>0}$, and $a\in\Z_{\geq0}$ we have
\[
\nu_{j}(a)
=\begin{cases}
0&\text{if $a>j$,}\\
q^{-2a}q&\text{if $a=j$,}\\
q^{-2a}(q-q^{-1})&\text{if $j>a>0$,}\\
q^{-2a}(1-q^{-1})&\text{if $a=0$.}
\end{cases}
\]
Moreover,
\[
\nu_{0}(a)
=\begin{cases}
0&\text{if $a>0$,}\\
1&\text{if $a=0$.}
\end{cases}
\]
\end{thm}

In light of \cref{heurcomp}, we make the following heuristic conjecture.

\begin{conj}\label{wemadeaconj}
If $p=3$, $\epsilon \in \{1,2\}$, and $a\in\Z_{\geq0}$, then
\[
\lim_{g\to\infty}{\mu_{\epsilon,g} '(a)}=
\begin{cases}
1-q^{-1}&\text{if }a=0\\
q^{-2a+1}\lp1-q^{-2}\rp&\text{if }a>0.
\end{cases}
\]
\end{conj}

\subsection{Results on the moduli space of hyperelliptic curves}\label{imoduli}

As above, let $K$ be an algebraic closure of $\Fq$.
For any positive integer $g$, let $\cH_g$ denote the moduli space of smooth hyperelliptic curves of genus $g$ over $K$; for any non-negative integer $a$, let $\cH^a_g$ the open sublocus of $\cH_g$ consisting of hyperelliptic curves with $a$-number $a$.
Using the above counting results we are able to approximate the number of $\Fq$-rational points on $\cH_g^a$ to sufficient accuracy to determine their dimensions.
In particular, we can derive the following theorem.

\begin{thm}\label{thm:strata}
Suppose that $p=3$, that $g\in\Z_{>0}$, and that $a\in\Z_{\geq0}$.
Write $\lceil\cdot\rceil$ for the least integer function.
Then
\[
\codim_{\cH_g}\lp \cH_g^a \rp=
\begin{cases}
2g-1&\text{if }a>\lceil\frac{g}{3}\rceil,\\
2a-1&\text{if }0<a\leq\lceil\frac{g}{3}\rceil,\\
0&\text{if }a=0.
\end{cases}
\]
In fact, if $a > \lceil\frac{g}{3}\rceil$, then $\cH_g^a$ is empty. 
\end{thm}

\begin{proof}
The second statement is simply a rewriting of 
\hyperref[squarefreevanishing]{\cref*{actualvalues} (\ref*{squarefreevanishing})}.
Thus, since $\cH_g$ is $2g-1$ dimensional, we may assume $a\leq\lceil\frac{g}{3}\rceil$.

To prove the first statement, we note that the cases of genus $1, 2, 3,$ and $4$ have been verified by direct calculation as carried out in \cite{ssg1}, \cite{ssg2}, \cite{ssg3}, and \cite{ssg4} respectively.
For $g >4$ we follow the method of proof of Prop 7.1 of~\cite{BG}, counting hyperelliptic curves weighted by the size of their automorphism group.
Using the terminology of \cite{KS}, let $\cH_g^a(\Fq)$ be the set of those curves in $\cH_g^a$ that are defined over $\Fq$, and define the \emph{intrinsic cardinality} of $\cH_g^a(\Fq)$ to be
\[
\IC(\cH_g^a(\Fq)) = \sum_{H \in \cH_g^a(\Fq)}{\left|\Aut_\Fq(H)\right|^{-1}}.
\]
By \cite[Theorem 10.7.5]{KS}, the number of $\Fq$-points on $\cH_g^a$ is $\IC(\cH_g^a(\Fq))$.
Thus, if we show that
\[
\IC(\cH_g^a(\Fq))=
\begin{cases}
q^{2g-2a}+O\lp q^{2g-2a-1}\rp&\text{if }0<a\leq\lceil\frac{g}{3}\rceil\\
q^{2g-1}+O\lp q^{2g-2}\rp&\text{if }a=0,
\end{cases}
\]
we may apply Deligne's proof of the Weil Conjectures \cite{D} to conclude that
\[
\dim{\lp\cH_g^a\rp}=
\begin{cases}
2g-2a&\text{if }0<a\leq\lceil\frac{g}{3}\rceil\\
2g-1&\text{if }a=0;
\end{cases}
\]
the first statement will then follow since $\cH_g$ is $2g-1$ dimensional.

Let $S$ be the set of squarefree polynomials in $\Fq[X]$ of degree $2g+1$ or $2g+2$, so for that every hyperelliptic curve $H\in\cH_g^a(\Fq)$ there exists $f\in S$ such that $H_f$ is a model of $H$.
Of course, the set $S$ is acted upon by the group of $\Fq$-isomorphisms of such curves, 
all of which are given by a linear fractional map on $X$ and a scalar multiple of $Y$ (see \cite{N} for details).
Let $G$ be this group, so that $|G|=(q-1)\left|\PGL(2,q)\right| =(q-1)(q^3-q)$.

We claim that for any $\alpha\in\lp\Fq\rp^\times$ and $\epsilon\in\lb1,2\rb$,
\[
\lv\lb f\in\alpha\cPe '(g)
\mid a(f)=a\rb\rv
=\lp q^{2g+\epsilon}-q^{2g+\epsilon-1}\rp\mu_{\epsilon,g}'(a).
\]
Indeed, the element of $G$ mapping $X\mapsto X$ and $Y\mapsto\sqrt{\alpha}$ induces an $a$-number preserving 
bijection between $\alpha\cPe '(g)$ and $\cPe '(g)$ (in fact, the $a$-number is invariant under any $\overline{\Fq}$-isomorphism, see \cref{collins}).
Thus, as elements of $G$ preserve $a$-numbers, we apply the orbit-stabilizer theorem to see that
\begin{align*}
\IC(\cH_g^a(\Fq))
&=\frac{1}{\lv G\rv}\lv\lb f\in S\mid a(f)=a\rb\rv\\
&=\frac{1}{(q-1)(q^3-q)}
\lp\sum_{\alpha\in\lp\Fq\rp^\times}{\lp q^{2g+1}-q^{2g}\rp\mu_{1,g}'(a)}
+\sum_{\alpha\in\lp\Fq\rp^\times}{\lp q^{2g+2}-q^{2g+1}\rp\mu_{2,g}'(a)}\rp\\\\
&=\frac{q-1}{(q-1)(q^3-q)}
\lp\lp q^{2g+1}-q^{2g}\rp\mu_{1,g}'(a)
+\lp q^{2g+2}-q^{2g+1}\rp\mu_{2,g}'(a)\rp\\
&=\lp\frac{q^{2g-1}}{1+q^{-1}}\rp\left(q^{-1}\mu_{1,g}'(a)+\mu_{2,g}'(a) \right).
\end{align*}
Next, we apply \cref{boundedvalues}~\hyperref[corpart1]{(\ref*{corpart1})} and~\hyperref[corpart2]{(\ref*{corpart2})} to deduce
\[
\IC(\cH_g^a(\Fq))
=\lp\frac{q^{2g-1}}{1+q^{-1}}\rp\left(q^{-1}\mu_{1,g}'(a)+\mu_{2,g}'(a) \right)=
\begin{cases}
q^{2g-2a}+O\lp q^{2g-2a-1}\rp&\text{if }0<a\leq\left [\frac{g}{3}\right ]\\
q^{2g-1}+O\lp q^{2g-2}\rp&\text{if }a=0,
\end{cases}
\]
as desired.
\end{proof}

\begin{cor}
In characteristic 3, if the strata $\cH_g^a$ are pure and non-empty, then they are irreducible. 
\end{cor}

Note that the for $g \leq 4$ the calculations of \cite{ssg1}, \cite{ssg2}, \cite{ssg3}, and \cite{ssg4} show that the non-empty a-number strata in characteristic 3 are indeed irreducible, and so it seems reasonable to conjecture that this continues to be the case for all larger genera.

\section*{Acknowledgments}\label{Acknowledgements}

The authors would like to thank Rachel Pries, Jeff Achter, Joseph Gunther, Avinash Kulkarni, and Adam Logan for very helpful conversations. 

\bibliography{RandomCartier}
\bibliographystyle{amsalpha}

\end{document}